\documentclass[12pt]{amsart}
\usepackage{amssymb,latexsym}
\newtheorem{theorem}{Theorem}
\newtheorem{lemma}[theorem]{Lemma}

\newtheorem{corollary}[theorem]{Corollary}

\newtheorem{conjecture}[theorem]{Conjecture}
\newtheorem{proposition}[theorem]{Proposition}

\def\RR{{\mathbb R}}

\def\se{{\subseteq}}

\title[Filling in arithmetic groups]{Filling boundaries of coarse manifolds in semisimple and solvable arithmetic groups}

\thanks{The authors gratefully acknowledge the support of the National Science Foundation.}

\author{Mladen Bestvina, Alex Eskin, \& Kevin Wortman}

\begin{document}

\begin{abstract} 
We provide partial results towards a conjectural generalization of a theorem of Lubotzky-Mozes-Raghunathan for arithmetic groups (over number fields or function fields) that implies, in low dimensions, both polynomial isoperimetric inequalities and finiteness properties.

As a tool in our proof, we establish polynomial isoperimetric inequalities and finiteness properties for certain solvable groups that appear as subgroups of parabolic groups in semisimple groups, thus generalizing a theorem of Bux.

We also develop a precise version of reduction theory for arithmetic groups whose proof is, for the most part, independent of whether the underlying global field is a number field or a function field.
\end{abstract}

\maketitle

Our main result is Theorem~\ref{t:p} below. Before stating it, we provide some background.

\subsection{Arithmetic groups} Let $K$ be a global field (number field or function field), and let $S$ be a set of finitely many inequivalent valuations of $K$ including one from each class of archimedean valuations. The ring $\mathcal{O}_S \subseteq K$ will denote  the corresponding ring of $S$-integers.

For any $v \in S$, we let $K_v$ be the completion of $K$ with respect to $v$ so that $K_v$ is a locally compact field. 

Let $\bf G$ be a noncommutative, absolutely almost simple, $K$-isotropic $K$-group. Let $G$ be the semisimple Lie group $$G=\prod _{v\in S} {\bf G}(K_v)$$ endowed with a left-invariant metric. Notice that $|S|$ is the number of simple factors of $G$.

Under the diagonal embedding, the arithmetic group  ${\bf G} (\mathcal{O}_S)$ is a lattice in $G$. The lattice being noncocompact is equivalent to the assumption that $\bf G$ is $K$-isotropic. The metric on $G$ restricts to a metric on ${\bf G} (\mathcal{O}_S)$.

Denote the Euclidean, or geometric, rank of $G$ by $k({\bf G},S)$, so that $$k({\bf G},S)=\sum_{v\in S}\text{rank}_{K_v}\bf G$$

\subsection{Word metric for higher rank arithmetic groups}
If $k({\bf G},S)\geq 2$, then ${\bf G} (\mathcal{O}_S)$ is well-known to be finitely generated. Thus, ${\bf G} (\mathcal{O}_S)$ can be endowed with a proper left-invariant word metric.  Lubotzky-Mozes-Raghunathan \cite{L-M-R} showed that the word metric is determined by $G$ by proving

\begin{theorem}\label{t:lmr}
The diagonal embedding ${\bf G} (\mathcal{O}_S) \hookrightarrow G$ is a quasi-isometric embedding when $k({\bf G},S)\geq 2$.
\end{theorem}

Bux-Wortman conjectured a natural generalization to the above theorem in \cite{B-W1}.
We introduce a slight reformulation of that conjecture as Conjecture~\ref{c:p} below. The reformulated version better illustrates the lack of dependence on whether the local fields $K_v$  are archimedean.

Before stating the conjecture, we must introduce the notion of a coarse manifold.

\bigskip

\subsection{Coarse manifolds} A \emph{coarse manifold} $\Sigma$ in a metric space $X$ is a function from the vertices of a triangulated manifold $M$ into $X$. In a slight abuse of language, we refer to the image of a coarse manifold as a coarse manifold, thus a coarse manifold in $X$ will be regarded as a subset of $X$. 

Given a coarse manifold $\Sigma$, we define $\partial \Sigma$ as the restriction of the function defining $\Sigma$ to  $\partial M$.

We say $\Sigma$ has \emph{scale} $r>0$ if every pair of adjacent vertices in $M$ maps to within distance $r$ of each other in $X$. We define the \emph{volume} of $\Sigma$ to be the number of vertices in $M$.

If $M$ is an $n$-manifold, we call $\Sigma$ a coarse $n$-manifold.
If $\Sigma '$ is a coarse manifold as well whose domain is the triangulated manifold $M'$, then we say that $\Sigma$ and $\Sigma '$ have the \emph{same topological type} if $M$ and $M'$ are homeomorphic.

\bigskip

\subsection{Expanding on Lubotzky-Mozes-Raghunathan} Having introduced the proper terminology, we state

\begin{conjecture}\label{c:p}
Given ${\bf G} (\mathcal{O}_S)$ as above and a scale factor $r_1$, there exists a linear polynomial $f$ and a scale factor $r_2$ such that
if $\Sigma  \subseteq G$ is a coarse $n$-manifold of scale $r_1$, with $\partial \Sigma  \subseteq {\bf G} (\mathcal{O}_S)$, and $n < k({\bf G},S)$, then there is a coarse $n$-manifold $\Sigma '  \se  {\bf G} (\mathcal{O}_S) $ of scale $r_2$, with the same topological type as $\Sigma$, and  such that  $\partial \Sigma ' = \partial \Sigma$ and $ vol(\Sigma ' ) \leq f(vol (\Sigma )) $.
\end{conjecture}

The bound of $n<k({\bf G},S)$ is known to be sharp in many cases.  Indeed, Bux-Wortman showed the bound is sharp when $K$ is a function field \cite{B-W1},
 Taback showed it is sharp when ${\bf G}(\mathcal {O}_S) = {\bf SL_2}(\mathbb{Z}[1/p])$ \cite{T}, and Wortman showed it was sharp if every place in $S$ is archimedean and the $K$-type of $\bf G$ is $A_n$, $B_n$, $C_n$, $D_n$, $E_6$ or $E_7$ \cite{W1}. The bound was conjectured to be sharp in general in \cite{B-W1}.

\bigskip

Notice that Lubotzky-Mozes-Raghunathan (Theorem~\ref{t:lmr}) would quickly follow from Conjecture~\ref{c:p}. Indeed, 
let $\gamma _1,\gamma_2 \in {\bf G}(\mathcal{O}_S)$. Because $G$ is quasi-isometric to a product of symmetric spaces and Euclidean buildings, there exists $r_1>0$, $L\geq1$, and $C \geq 0$ (that depend only on $G$) and a coarse path $\Sigma \se G$ of scale $r_1$ such that $\partial \Sigma = \{\, \gamma _1 , \gamma _2 \,\}$ and the volume of $\Sigma$ is bounded above by $Ld(\gamma _1, \gamma _2) +C$. We let $r_2$ and $f$ be as in Conjecture~\ref{c:p}, so there is a coarse path $\Sigma \se {\bf G}(\mathcal{O}_S)$ of scale $r_2>0$ and volume bounded above by $f(Ld(\gamma _1, \gamma _2) +C)$. We may assume the finite generating set of ${\bf G}(\mathcal{O}_S)$ contains all elements of ${\bf G}(\mathcal{O}_S)$ whose distance from $1$ is less than $r_2$, so the theorem follows.

\subsection{Isoperimetic inequalities} 
Recall that a group $\Gamma$ is of type $F_n$ if there is a $K(\Gamma ,1)$ with finite $n$-skeleton. 

If a group $\Gamma$ is of type $F_n$, then we let $X$ be an $(n-1)$-connected CW-complex that $\Gamma$ acts on cellularly, properly, and cocompactly. Suppose $1\leq m \leq n-1$. If there are constants $L, t \geq 1$ and $C\geq 0$ such that for any cellular $m$-sphere $\Sigma \se X$ there is a cellular $(m+1)$-disk $D \se X$ such that  $\partial D =\Sigma$ and $\text{vol}(D)\leq L \text{vol}(\Sigma)^t +C$, then we say that $\Gamma$ satisfies a \emph{polynomial $m$-dimensional  isoperimetric inequality}. (Here the volume of $\Sigma$ and $D$ are the number of cells that they contain.)

If in the above $t$ can be taken to be $\frac{m+1}{m}$, then we say that $\Gamma$ satisfies a \emph{Euclidean $m$-dimensional  isoperimetric inequality}. 

Satisfying a polynomial or Euclidean $m$-dimensional  isoperimetric inequality is well-known to be a quasi-isometry invariant, so it is independent of the choice of the space $X$.

\bigskip

Using a similar argument to that which was given in the proof above that Conjecture~\ref{c:p} implies Theorem~\ref{t:lmr}, we can take any coarse sphere $\Sigma \se {\bf G}(\mathcal{O}_S)$ of dimension $m \leq k({\bf G},S) -2$, find a coarse $(m+1)$-disk  in $G$ whose boundary is $\Sigma$ and whose volume is Euclidean with respect to $\Sigma$ ($G$ is quasi-isometric to a product of symmetric spaces and Euclidean buildings so this is always possible), and then use Conjecture~\ref{c:p} to find a corresponding coarse $(m+1)$-disk in ${\bf G}(\mathcal{O}_S)$ whose boundary is $\Sigma$ and whose volume is Euclidean with respect to $\Sigma$. This brief sketch of a proof will be made precise in Section~\ref{s:finisht}, and it proves that  Conjecture~\ref{c:p} would imply

\begin{conjecture}\label{c:iso}
 ${\bf G} (\mathcal{O}_S)$ satisfies a Euclidean $m$-dimensional isoperimetric inequality  if $m \leq 
k({\bf G},S) -2 $. In particular, the Dehn function for ${\bf G} (\mathcal{O}_S)$ is quadratic if $
k({\bf G},S) \geq 3$.
\end{conjecture}

Thurston's conjecture that ${\bf SL_4}(\mathbb{Z})$ has a quadratic Dehn function is a special case of Conjecture~\ref{c:iso} since $\text{rank}_{\mathbb{R}}{\bf SL_4}=3$.

As evidence for Conjecture~\ref{c:iso}, Dru\c{t}u proved that  ${\bf G} (\mathcal{O}_S)$ has a Dehn function that is bounded above by the function $x \mapsto x^{2+\epsilon}$ for any $\epsilon >0$ if $S$ contains only archimedean valuations, the $K$-rank of $\mathbf{G}$ equals $1$, and  $k({\bf G},S) \geq 3$ \cite{D}.

Young proved that if ${\bf G} (\mathcal{O}_S)=\bf{SL_n}(\mathbb{Z})$, then ${\bf G} (\mathcal{O}_S)$ has a quadratic Dehn function if $n \geq 5$  \cite{Y}. The condition $n \geq 5$ implies  $k({\bf G},S) \geq 4$.

The work of  Dru\c{t}u and Young are the only results in the literature that establish polynomial $m$-dimensional isoperimetric  inequalities for noncocompact arithmetic groups when $m \leq 
k({\bf G},S) -2 $.

\bigskip

\subsection{Main result} The main result proved in this paper is partial progress in proving Conjecture~\ref{c:p}. Namely,

\begin{theorem}\label{t:p}
Given ${\bf G} (\mathcal{O}_S)$ as above and a scale factor $r_1$, there exists a polynomial $f$  and a scale factor $r_2$ such that
if $\Sigma  \subseteq G$ is a coarse $n$-manifold of scale $r_1$, with $\partial \Sigma  \subseteq {\bf G} (\mathcal{O}_S)$, and $n < |S|$, then there is a coarse $n$-manifold $\Sigma ' \se  {\bf G} (\mathcal{O}_S) $ of scale $r_2$, with the same topological type as $\Sigma$, and such that  $\partial \Sigma ' = \partial \Sigma$ and $ vol(\Sigma ' ) \leq f(vol (\Sigma )) $.

If $n=1$, then $f$ can be taken to be linear.
\end{theorem}

If $\overline{K}$ is the algebraic closure of $K$, then for any $v \in S$ we have  $\text{rank}_{K_v}{\bf G} \leq \text{rank}_{\overline{K}}\bf G$, and it is a consequence of $\bf G$ being $K$-isotropic that $1 \leq \text{rank}_{K_v}\bf G$ for all $v \in S$. (In other words, each simple factor of $G$ has positive Euclidean rank.) Therefore,  $$|S|\leq k({\bf G},S)=\sum_{v\in S}\text{rank}_{K_v}{\bf G} \leq  |S|\text{rank}_{\overline{K}}\bf G$$ and the inequalities above are sharp.
 
Applying the argument above that Conjecture~\ref{c:p} implies Lubotzky-Mozes-Raghunathan (Theorem~\ref{t:lmr}), we see that Theorem~\ref{t:p} implies Lubotzky-Mozes-Raghunathan for those arithmetic groups for which $|S|\geq 2$.

In higher dimensions --- and similar to the reasoning above that Conjecture~\ref{c:p} implies Conjecture~\ref{c:iso} --- Theorem~\ref{t:p} implies 

\begin{corollary}\label{c:isop}
 ${\bf G} (\mathcal{O}_S)$  satisfies a polynomial $m$-dimensional isoperimetric inequality when $m \leq |S|-2$.
 In particular, the Dehn function of  ${\bf G} (\mathcal{O}_S)$ is bounded above by a polynomial if $|S|\geq 3$.
\end{corollary}

\bigskip

\subsection{Finiteness properties} One cannot inquire about the word metric of a group if the group in question is not finitely generated.
Similarly, $m$-dimensional isoperimetric inequalities only make sense for groups that are of type $F_{m+1}$. Thus, in order for Theorem~\ref{t:lmr}, Conjecture~\ref{c:iso}, and Corollary~\ref{c:isop} to be well-posed, we need to know that ${\bf G} (\mathcal{O}_S)$ is of type $F_{ k({\bf G},S) -1}$, and this is known to be true. Indeed, Raghunathan proved that ${\bf G} (\mathcal{O}_S) $ is of type $F_{n}$ for all $n$ when $S$ consists of only archimedean places \cite{R}, Borel-Serre established that ${\bf G} (\mathcal{O}_S) $ is of type $F_{n}$ for all $n$ when $K$ is a number field \cite{B-S}, and Bux-Gramlich-Witzel recently established that  ${\bf G} (\mathcal{O}_S)$ is of type $F_{ k({\bf G},S) -1}$ in the case when $K$ is a function field  \cite{B-G-W}. 

But while the finiteness properties of ${\bf G} (\mathcal{O}_S) $ that are needed for Theorem~\ref{t:lmr} and Corollary~\ref{c:isop} to be well-posed are known --- that ${\bf G} (\mathcal{O}_S)$ is of type $F_{ |S|-1} $ --- our proof makes no use of these finiteness properties, not even of finite generation.
 Rather, the needed finiteness properties can be derived as a corollary that follows quickly from Theorem~\ref{t:p}.

We illustrate here a quick proof that Theorem~\ref{t:p} implies that  ${\bf G} (\mathcal{O}_S)$ is of type $F_{ |S|-1} $: Suppose ${\bf G} (\mathcal{O}_S)$ and $r_1>0$ are given. For $s>0$ we let  $R(s)$ be the simplicial complex formed by declaring $(k+1)$-tuples of points in ${\bf G} (\mathcal{O}_S)$ to be a simplex if each pair of points in the $(k+1)$-tuple are within distance $s$ of each other. Then $R(\infty)$ is contractible, and the natural action of ${\bf G} (\mathcal{O}_S)$ on $R(\infty)$ has finite cell stabilizers. Let  $m \leq |S| -2$. Any $m$-sphere in $R(r_1)$ corresponds naturally to a coarse $m$-sphere  in ${\bf G} (\mathcal{O}_S)$ of scale $r_1$, and Theorem~\ref{t:p} implies the existence of an $(m+1)$-disk in $R(r_2)$ that fills that sphere. Thus ${\bf G} (\mathcal{O}_S)$ is of type $F_{ |S| -1}$ by  Brown's criterion (see e.g. Theorem 7.4.1 \cite{Ge}). 

Notice that the proof in the previous paragraph does not use Theorem~\ref{t:p} in its fullest, as 
 the volumes of the filling disks used in the proof are irrelevant.
 
 Using a similar proof as above, Conjecture~\ref{c:p} would imply  that ${\bf G} (\mathcal{O}_S)$ is of type $F_{ k({\bf G},S) -1}$. Again, this result is known by work of 
Raghunathan, Borel-Serre,  and Bux-Gramlich-Witzel, and it is known by Bux-Wortman that the group  ${\bf G} (\mathcal{O}_S)$ is not of type $F_{ k({\bf G},S)} $ when $K$ is a function field  \cite{B-W1}. Thus, Conjecture~\ref{c:p} would imply the strongest possible finiteness result for ${\bf G} (\mathcal{O}_S)$ that is independent of whether the global field $K$ is a number field or a function field.

\subsection{Solvable groups} Our proof of Theorem~\ref{t:p} proceeds by first studying the large scale geometry of certain solvable groups. We prove the following generalization of Gromov's result that certain solvable Lie groups of the form $\RR ^{n-1} \ltimes \RR ^n$ satisfy a quadratic Dehn function if $n \geq 3$:

\begin{theorem}\label{p:s}
Let $\bf Q$ be a proper $K$-parabolic subgroup of $\bf G$. Let $\bf U_Q$ be the unipotent radical of $\bf Q$ and let $\bf A_Q$ be the maximal $K$-split torus in the center of a $K$-Levi subgroup of $\bf Q$.

 Given $r_1>0$, there exists $r_2>0$ and a polynomial $f$ such that any coarse $m$-sphere $ \Sigma \se ({\bf U_Q A_Q}) (\mathcal{O}_S)$ of scale $r_1$ can be realized as the boundary of a coarse $(m+1)$-ball in  $({\bf U_Q A_Q}) (\mathcal{O}_S)$ of scale $r_2$ whose volume is bounded above by $f (\emph{vol}(\Sigma))$ as long as $m \leq |S|-2$.
 
 In particular, $({\bf U_Q A_Q}) (\mathcal{O}_S)$ is of type $F_{|S|-1}$ and satisfies a polynomial $m$-dimensional  isoperimetric inequality if $m \leq |S|-2$.
 \end{theorem}

As a special case of the above proposition, if $\bf G$ is $K$-split and $\bf B$ is a $K$-defined Borel subgroup of  $\bf G$, then ${\bf U_B A_B}= \bf B$, so ${\bf B}(\mathcal{O}_S)$ is of type $F_{|S|-1}$. Thus, Theorem~\ref{p:s} generalizes ``half" of Bux's theorem \cite{Bu}:

\begin{theorem}
Suppose $K$ is a function field, that $\bf G$ is $K$-split, and that ${\bf B} \leq {\bf G}$ is a $K$-defined Borel subgroup. Then ${\bf B}(\mathcal{O}_S)$ is of type $F_{|S|-1}$ but not of type $F_{|S|}$.
\end{theorem}

Bux's theorem is proved using piecewise linear Morse theory. It is the most prominent result in the mostly unexplored field of finiteness properties of solvable arithmetic groups over function fields.

 Wortman proves a converse to Theorem~\ref{p:s} by showing that $({\bf U_Q A_Q}) (\mathcal{O}_S)$ is not of type $F_{|S|}$ if $K$ is a function field, and that $({\bf U_Q A_Q}) (\mathcal{O}_S)$ has an exponential $(|S|-1)$-dimensional Dehn function if $K$ is a number field \cite{W3}, thus generalizing the ``other half" of Bux's theorem.

\bigskip

 \subsection{Outline of proof}
The plan for our proof was motivated by the unpublished Abels-Margulis proof of the Lubotzky-Mozes-Raghunathan theorem. 

Section~\ref{s:p} of this paper contains some preliminary material and notation, and Section~\ref{s:ex} displays an example that readers can use to guide themselves through the proofs in this paper.

In Section~\ref{s:red} we state the precise version of reduction theory (Proposition~\ref{p:prune}) that we will use in our proof of Theorem~\ref{t:p}. We give a proof of  Proposition~\ref{p:prune} in an appendix: Section~\ref{s:a}. Aside from starting with the well-known result that there are finitely many equivalence classes of minimal $K$-parabolic subgroups of $\bf G$ modulo ${\bf G}(\mathcal{O}_S)$, our proof is independent of the characteristic of $K$.

The proof of our main result, Theorem~\ref{t:p}, follows quickly from reduction theory (Proposition~\ref{p:prune}) and Proposition~\ref{p:bound} which states that ``boundaries of parabolic regions" have nice filling properties. In a first reading, the reader may wish to read the statement of Proposition~\ref{p:bound} from Section~\ref{s:bound}, along with Section~\ref{s:red}, before proceeding to Section~\ref{s:proof} for a short proof of our main result.

Section~\ref{s:sol} contains a proof of our theorem on fillings in solvable arithmetic groups, Theorem~\ref{p:s}, which is equivalently stated as Proposition~\ref{p:solvable}. It's used in Section~\ref{s:bound} to prove our result on filling in boundaries of parabolic regions, Proposition~\ref{p:bound}.

Section~\ref{s:finisht} contains a short proof that our main result implies the isoperimetric inequalities stated in Corollary~\ref{c:isop}.

 \subsection{Acknowledgements} The original blueprint for this proof was developed at an AIM workshop in September 2008. We would like to thank AIM as well as the co-organizers of that workshop, Nathan Broaddus and Tim Riley, and the rest of the participants. We would also like to thank Shahar Mozes for his helpful conversations during that workshop.
 
 We also thank Kai-Uwe Bux, Brendan Kelly, Amir Mohammadi, Dave Morris,  and Robert Young for helpful conversations.

\section{Preliminaries.}\label{s:p}

Let $K$, $\mathcal{O}_S$, and $\bf G$  be as above. Because $\bf G$ is $K$-isotropic, it has  a minimal $K$-parabolic subgroup $\bf{P}$.  Let  $\bf A$ be a maximal $K$-torus in $\bf P$.

We denote the root system for $({\bf G},{\bf A})$ by $\Phi$. A positive set $\Phi ^+$ is defined by $\bf P$. We let $\Delta \se \Phi^+$ be the set of simple roots.

For $I \subseteq \Delta$, we let $[I] \se \Phi$ be the linear combinations generated by $I$, and we let $\Phi (I)^+=\Phi ^+ - [I]$ and $[I]^+=[I] \cap \Phi ^+$.

If $\alpha \in \Phi $, we let ${\bf U}_{(\alpha)} $ be the root group corresponding to $\alpha$. For any set $\Psi \se \Phi ^+$ that is closed under addition, we let ${\bf U}_\Psi$ be the group $\prod _{\alpha \in \Psi} {\bf U}_{(\alpha)}$. The group $\prod _{v \in S}{\bf U}_{\Psi}(K_v)$ is naturally identified with a product of topological vector spaces that we endow with a norm $|| \cdot ||$.

If $I \subseteq \Delta$, then we let ${\bf A}_I$ be the connected component of the identity in $(\cap _{\alpha \in I} \text{Ker}(\alpha))$.
We let ${\bf Z_G}({\bf A}_I)$ be the centralizer of ${\bf A}_I$ in $\bf G$ so that ${\bf Z_G}({\bf A}_I)={\bf M}_I{\bf A}_I$ where ${\bf M}_I$ is a reductive $K$-group with $K$-anisotropic center. We denote by ${\bf P}_I$ the parabolic group ${\bf U}_{\Phi(I)^+}{\bf M}_I{\bf A}_I$. The Levi subgroup ${\bf M}_I{\bf A}_I$ normalizes the unipotent radical ${\bf U}_{\Phi(I)^+}$, and elements of ${\bf A}_I$ commute with those of ${\bf M}_I$

Note that if $\alpha \in \Delta$, then ${\bf P}_{\Delta -\alpha}$ is a maximal proper $K$-parabolic subgroup of $\bf{G}$, and that ${\bf P}_\emptyset = {\bf P}$. To ease notation a bit, we will also denote $ {\bf U}_{\Phi(\emptyset)^+}={\bf U}_{\Phi^+}$, ${\bf M}_\emptyset $, and ${\bf A}_\emptyset $ at times as  ${\bf U} $, ${\bf M} $, and ${\bf A}$ respectively.

In the remainder of this paper we denote the product over $S$ of local points of a $K$-group by ``unbolding", so that, for example, $$G=\prod _{v\in S}{\bf G}(K_v)$$

\subsection{The metric on $G$}
Suppose $S=\{v\}$. Let $\bf Q$ be a minimal $K_v$-parabolic subgroup of $\bf G$ with maximal $K_v$-torus ${\bf A}_v$ and unipotent radical ${\bf U}_v$. Then there is a compact set $B \se G$ such that $U_vA_vB=G$ and thus the left invariant metric on $U_vA_v$ is quasi-isometric to $G$. 

It follows that  $A_v$ with the restricted metric from $G$ is quasi-isometric to Euclidean space. Also, if $u \in U_v$ then there is some $L>0$ such that $(1/L)\log (||u||+1) \leq d(1,u) \leq L \log (||u||+1)$. The properties of the metric on $G$ that we will use in this paper are deduced from this paragraph after taking the product metric in the case when $|S|>1$.

\subsection{Bruhat decompostion}
We let $W \se {\bf G}(K)$ be a set of coset representatives, including 1, for the Weyl group ${\bf N_G(A)} / {\bf Z_G(A)}$ where ${\bf N_G(A)} $ is the normalizer of $\bf A$ in $\bf G$. Then ${\bf G}(K)$ is a disjoint union $\coprod _{w\in W} {\bf P}(K)w{\bf P}(K)$.

 \subsection{Conjugation} If $g,h \in G$ and $H \se G$, then we denote $ghg^{-1}$ as $^g h$ and $gHg^{-1}$ as $^g H$. 
 
\subsection{Bounds} Throughout, we write $a=O(c)$ to mean that there is some constant $\kappa$ depending only on $G$ and  ${\bf G}(\mathcal{O}_S)$  such that $a \leq \kappa c$.
 
 \section{An example to follow throughout}\label{s:ex}
 
 In this section we provide an example of an arithmetic group ${\bf G}(\mathcal{O}_S)$ that those less familiar with arithmetic groups may prefer to focus on while reading the rest of this paper. The example we provide is the arithmetic group ${\bf SL_3}(\mathbb{Z}[1/p_1,\ldots , 1/p_k])$ where $p_1,\ldots , p_k$ are prime numbers. It is an example that is simple enough that most of this paper can be read with it in mind and without any knowledge of the general properties of semisimple groups, but it is complicated enough that it still illustrates all of the important features and techniques of our general proof.
 
 Although we do provide explicit examples in this section of $K$, $\bf G$,  ${\bf G}(\mathcal{O}_S)$, etc., these examples are particular only to this section, and nowhere in the remainder of the paper is any part of our proof restricted to this particular example.
 
 \subsection{Global field, valuations, and $S$-integers}
 For our example we take $K$ to be the field of rational numbers $\mathbb{Q}$. 
 
 Let $v_\infty$ be the archimedean valuation on $\mathbb{Q}$. That is, ${v_\infty}$ endows $\mathbb{Q}$ with the restriction of the standard norm on the real numbers, and the completion $K_{v_\infty}$ is the field of real numbers $\mathbb{R}$. 
 
 Given a prime number $p$, there is also a $p$-adic valuation on $\mathbb{Q}$, denoted $v_p$,  whose completion yields the locally compact field of $p$-adic numbers $\mathbb{Q}_p$. What's important to know about the $p$-adic norm is that the norm of $p$ is less than 1, and hence the $p$-adic norm of $1/p$ is greater than 1.
 
 The archimedean and $p$-adic valuations are the only valuations that exist for $\mathbb{Q}$ up to scale.

 We can take for a set of valuations $S$ to be $\{v_\infty , v_{p_1},\ldots, v_{p_k}\}$ where each $p_i$ is a distinct prime. Thus, $|S|=k+1$.
 
 The ring of $S$-integers $\mathcal{O}_S$ in this example is the ring $\mathbb{Z}[1/p_1,\ldots , 1/p_k]$. Notice that this ring has as units the elements of $\mathbb{Q}$ whose numerators and denominators are products of powers of the primes $p_1, \ldots , p_k$ and their negatives.
 
It's a good exercise to check that  $\mathbb{Z}[1/p_1,\ldots , 1/p_k]$ embedded diagonally into the product $\mathbb{R} \times \mathbb{Q}_{p_1} \times \cdots \times \mathbb{Q}_{p_k}$ is a discrete and cocompact subring.
 
\subsection{Simple group and resulting arithmetic group and semisimple Lie group} We take for our example of $\bf G$ the group ${\bf SL_3}(\mathbb{C})$ of $3\times 3$ matrices with entries in $\mathbb{C}$ whose determinants equal 1.

In what follows, if $R$ is a subring of $\mathbb{C}$, then ${\bf G}(R)$ is understood to be the group ${\bf SL_3}(R)$. In particular, the algebraic closures of $\mathbb{R}$ and $\mathbb{Q}_p$ are isomorphic to $\mathbb{C}$, so we can consider $\mathbb{R}$ and $\mathbb{Q}_p$ to be subrings of $\mathbb{C}$ and then the Lie group $G$ is the product $${\bf SL_3}(\mathbb{R})\times {\bf SL_3}(\mathbb{Q}_{p_1})\times \cdots \times {\bf SL_3}(\mathbb{Q}_{p_k})$$

The arithmetic group ${\bf G}(\mathcal{O}_S)$ in this example is  ${\bf SL_3}(\mathbb{Z}[1/p_1,\ldots , 1/p_k])$. Embedded diagonally into $G$, it is a discrete subgroup.

More generally, we regard elements of ${\bf G}(\mathbb{Q})={\bf SL_3}(\mathbb{Q})$ as elements of $G$ via the diagonal embedding. 
 
The number $\text{rank}_\mathbb{R} {\bf SL_3}(\mathbb{C})$ is the maximal dimension of a subgroup of ${\bf SL_3}(\mathbb{C})$ that is diagonal after being conjugated by an element of ${\bf SL_3}(\mathbb{R})$. Thus, $\text{rank}_\mathbb{R} {\bf SL_3}(\mathbb{C})=2$. Similarly, $\text{rank}_\mathbb{Q_p} {\bf SL_3}(\mathbb{C})=2$ for any prime $p$, so for our choice of ${\bf G}(\mathcal{O}_S)$ we have that $k({\bf G}, S) = \sum_{v \in S}\text{rank}_{K_v} {\bf SL_3}(\mathbb{C})= \sum_{v \in S}2=2|S|$.
 
\subsection{Roots} We let $\bf P$ be the upper-triangular subgroup of ${\bf SL_3}(\mathbb{C})$ whose entries below the diagonal all equal 0, and we let $\bf A$ be the 2-dimensional group of all diagonal matrices in ${\bf SL_3}(\mathbb{C})$. The choice of $\bf P$ and $\bf A$ provides us with  the set of six roots  $$\Phi=\{\, \lambda _{ij} \mid 1 \leq i,j \leq 3 {\text{ and }}  i \neq j \,\}$$ where each  $\lambda _{ij} : \mathbf{A} \rightarrow \mathbb{C}^\times $ is defined by $\lambda _{ij}(a_1,a_2,a_3)=a_i / a_j$ if $(a_1, a_2, a_3 )\in { \bf A}$ is the matrix whose 3 diagonal entries are given by $a_1$, $a_2$, and $a_3$ respectively.
 
The operation of pointwise multiplication of roots is written additively, so that for example, $\lambda_{12}+\lambda_{23}=\lambda _{13}$ and $- \lambda_{ij}= \lambda_{ji}$. With this structure $$\Phi=\{\, \lambda _{12},\, \lambda _{23},\, \lambda _{12}+ \lambda _{23},\,
 -\lambda _{12},\, -\lambda _{23},\, -\lambda _{12}- \lambda _{23}\, \}$$ The positive set of roots (consistent with our choice of $\bf P$) is $$\Phi^+=\{ \,\lambda _{12},\, \lambda _{23},\, \lambda _{12}+ \lambda _{23}\, \}$$ and the set of simple roots is $$\Delta=\{ \, \lambda _{12}, \, \lambda _{23}\,\}$$ The ``highest root" is $\lambda_{13}=\lambda _{12}+ \lambda _{23}$.

\subsection{Subgroups of $\bf G$ defined by roots}
The group ${\bf A}_{\{\lambda_{12}\}} \leq {\bf A}$ is given by the group of diagonal matrices $(a_1,a_2,a_3)$ where $a_1=a_2$.   The group ${\bf A}_{\{\lambda_{23}\}} \leq {\bf A}$ is given by the group of diagonal matrices $(a_1,a_2,a_3)$ where $a_2=a_3$.     The group ${\bf A}_{\emptyset}$ equals $ {\bf A}$.

Notice that the multiplicative group ${\bf A}_{\{\lambda_{12}\}}(\mathcal{O}_S)$ is virtually isomorphic to $\mathbb{Z}^k$, as is ${\bf A}_{\{\lambda_{23}\}}(\mathcal{O}_S)$. The group ${\bf A}(\mathcal{O}_S)$ is virtually isomorphic to $\mathbb{Z}^{2k}$.

The group ${\bf M}_{\{\lambda_{12}\}}\cong {\bf SL_2}(\mathbb{C})$ is the set of matrices in ${\bf SL_3}(\mathbb{C})$ that can be written in the form 
\[ \begin{pmatrix}
* & * & 0  \\
* & * & 0   \\
0 & 0 & 1 
\end{pmatrix} 
\]  
 The group ${\bf M}_{\{\lambda_{23}\}}\cong {\bf SL_2}(\mathbb{C})$ consists of matrices in ${\bf SL_3}(\mathbb{C})$ that can be written in the form 
 \[\begin{pmatrix}
1 & 0 & 0  \\
0 & * & *   \\
0 & * & * 
\end{pmatrix} 
\]   The group ${\bf M}={\bf M}_{\emptyset}$ is trivial. For $I \subsetneq \Delta$, elements of ${\bf M}_I$ commute with elements of ${\bf A}_I$.

For any $\lambda_{ij} \in \Phi$, the root group ${\bf U}_{(\lambda _{ij})}$ is the subgroup of  ${\bf SL_3}(\mathbb{C})$ that equals the identity matrix in every entry except for perhaps the entry in the $i$-th row and $j$-th column. Notice that if $u \in {\bf U}_{(\lambda _{ij})}$ and $a \in {\bf A}$ then $aua^{-1}=\lambda_{ij}(a) u$.

Notice that ${\bf U}_{(\lambda _{ij})}(\mathbb{R})$ is isomorphic to the additive group $\mathbb{R}$, and thus ${\bf U}_{(\lambda _{ij})}(\mathbb{R})$ has a natural structure of a normed 1-dimensional vector space with an obvious choice of norm. Likewise ${\bf U}_{(\lambda _{ij})}(\mathbb{Q}_p)$ is isomorphic to the normed vector space $\mathbb{Q}_p$. The group ${ U}_{(\lambda _{ij})}$ is  isomorphic to the product $\mathbb{R} \times \mathbb{Q}_{p_1} \times \cdots \times \mathbb{Q}_{p_k}$.

Following the notation from the previous section, we have  that $\Phi(\{\lambda_{12}\})^+=\{\lambda_{13}, \lambda_{23}\}$, $\Phi(\{\lambda_{23}\})^+=\{\lambda_{12}, \lambda_{13}\}$, and $\Phi(\emptyset )^+=\{\lambda_{12}, \lambda_{13}, \lambda_{23}\}$. From this, and the fact that for $I \subsetneq \Delta $ the group $U_{\Phi(I)^+}$ is simply the product of those ${ U}_{(\lambda _{ij})}$ with $\lambda _{ij} \in \Phi(I)^+$, one can easily see that each $U_{\Phi(I)^+}$ has the topological structure of a product of normed vector spaces (each of dimension 2 or 3 depending on the cardinality of $I$) and we endow each $U_{\Phi(I)^+}$ with a norm that we denote simply as $|| \cdot ||$, ignoring the set $I$ in our notation for the norm. The group structure on $U_{\Phi(I)^+}$  is also of a product of vector spaces if $I=\{\lambda_{12}\}$ or if $I=\{\lambda_{23}\}$. If  $I=\emptyset$ then the group structure on ${\bf U}={\bf U}_{\Phi(\emptyset )^+}$ is nilpotent, but not abelian.

It is easy to form explicitly the parabolic groups ${\bf P}_I={\bf U}_{\Phi(I)^+}{\bf M}_I{\bf A}_I$ for $I \subsetneq \Delta$. The group ${\bf P}_{\{\lambda _{12}\}}$ is the set of matrices in 
${\bf SL_3}(\mathbb{C})$ of the form 
\[ \begin{pmatrix}
* & * & *  \\
* & * & *   \\
0 & 0 & * 
\end{pmatrix}= \begin{pmatrix}
1 & 0 & *  \\
0 & 1 & *   \\
0 & 0 & 1 
\end{pmatrix} 
\begin{pmatrix}
* & * & 0  \\
* & * & 0   \\
0 & 0 &  1
\end{pmatrix} 
\begin{pmatrix}
a_1 & 0 & 0  \\
0 & a_1 & 0   \\
0 & 0 & a_1^{-2} 
\end{pmatrix} 
\]  
Elements of  ${\bf P}_{\{\lambda _{23}\}}$ have the form 
\[ \begin{pmatrix}
* & * & *  \\
0 & * & *   \\
0 & * & * 
\end{pmatrix} 
\]  
and elements of ${\bf P}={\bf P}_{\{ \emptyset \}}$ have the form 
\[ \begin{pmatrix}
* & * & *  \\
0 & * & *   \\
0 & 0 & * 
\end{pmatrix} 
\]  

It's easy to check that ${\bf U}_{\Phi(I)^+}$ is a normal subgroup of ${\bf P}_I$.

Notice, that the inverse-transpose automorphism of ${\bf SL_3}(\mathbb{C})$ restricts to an isomorphism between ${\bf P}_{\{\lambda _{12}\}}$ and ${\bf P}_{\{\lambda _{23}\}}$. Much of the proof in this paper is considered by examining parabolic groups (or spaces associated with them) in the different cases enumerated by proper subsets $I \subsetneq \Delta$. Thus, when considering our proof as it applies to the particular example from this section, one can often restrict to just two cases: $I=\{\lambda_{12}\}$ and $I=\emptyset$.

\subsection{Parabolic regions}
In the next section, ``parabolic regions" will be defined. They will be denoted as $R_I$ for $I \subsetneq \Delta$. 


Very nearly, $R_I$ is the space  $$ {\bf P}_I(\mathbb{Z}[1/p_1,\ldots , 1/p_k])A_I^+$$ or equivalently $${\bf U}_{\Phi(I)^+}(\mathbb{Z}[1/p_1,\ldots , 1/p_k]){\bf M}_I(\mathbb{Z}[1/p_1,\ldots , 1/p_k]){A}_I^+$$
where ${ A}_I^+$ is defined in the next section as those $a \in A$ such that $a$ commutes with elements of ${\bf M}_I(\mathbb{Z}[1/p_1,\ldots , 1/p_k])$ and such that, up to multiplying $a$ by an element of ${\bf A}(\mathbb{Z}[1/p_1,\ldots , 1/p_k])$, $||a^{-1}ua|| \leq  ||u||$ for any  $u \in {\bf U}_{\Phi(I)^+}(\mathbb{Z}[1/p_1,\ldots , 1/p_k])\leq { U}_{\Phi(I)^+}$. 

Since ${\bf U}_{\Phi(I)^+}(\mathbb{Z}[1/p_1,\ldots , 1/p_k])$ is a cocompact lattice in $ { U}_{\Phi(I)^+}$, and because we are only interested in the large scale geometry of $R_I$, the actual defintion of a parabolic region that we will use is $$ R_I={ U}_{\Phi(I)^+}{\bf M}_I(\mathbb{Z}[1/p_1,\ldots , 1/p_k]){A}_I^+$$ We make this substitution only to ease notation a bit.

\subsection{Weyl group and cusps}
We let $W \se {\bf SL_3}(\mathbb{C})$ be the set of 6 matrices that permute the standard coordinate vectors $e_1,e_2,e_3 \in \mathbb{C}^3$. The elements of $W$ normalize $\bf A$, and they are representatives for the Weyl group which is defined as the normalizer of $\bf A$ modulo $\bf A$.

The ``longest element" of the Weyl group is represented by the transposition that interchanges $e_1$ and $e_3$ and fixes $e_2$.
 
The set $F$ from Theorem~\ref{t:finitecusps} below 
as it applies to the example illustrated in this section consists of only the identity element. That is,
the double coset space 
$${\bf SL_3}(\mathbb{Z}[1/p_1,\ldots , 1/p_k]) \backslash {\bf SL_3}(\mathbb{Q}) / {\bf P}(\mathbb{Q})$$
is a single point. Indeed, it is well known that $${\bf SL_3}(\mathbb{Z}) \backslash {\bf SL_3}(\mathbb{Q}) / {\bf P}(\mathbb{Q})$$
is a single point, as this is equivalent to the assertion  that ${\bf SL_3}(\mathbb{Z})$ acts transitively on complete flags in $\mathbb{Q}^3$.

 \subsection{End of example} We have now concluded our example, and in the remainder of the paper we will return to our more general notation where $K$ is an arbitrary global field, $\mathbf{G}$ is an arbitrary noncommutative, absolutely almost simple, $K$-isotropic $K$-group, $\bf P$ is an arbitrary minimal $K$-parabolic subgroup of $\bf G$, and so on.

 \section{Parabolic regions and the pruning of $G$ to ${\mathbf{G}}(\mathcal{O}_S)$}\label{s:red}

This section contains the precise statement from reduction theory that our proof requires. We begin by recalling the ``finiteness of cusps" theorem from reduction theory.

\begin{theorem}\label{t:finitecusps}
There is a finite set $F \se {\bf G}(K)$ of coset representatives for  ${\bf G}(\mathcal{O_S}) \backslash {\bf G}(K) / {\bf P}(K)$.
\end{theorem}

\begin{proof}
Restriction of scalars applied to Proposition 15.6 of Borel's book on arithmetic groups \cite{Boreltext} gives the result when $\mathcal{O}_S$ is the ring of integers in a number field. The general case for number fields is immediate since any ring of $S$-integers contains the ring of integers.

When $K$ is a function field, this theorem is the statement of Satz 8 in Behr's \cite{Be}. Behr's proof needs a technical hypothesis (used for Satz 5). However, Harder has removed the need for that hypothesis: Korollar 2.2.7 in \cite{H} can be used as a replacement for  Satz 5 \cite{Be} in the proof.

\end{proof}

\subsection{Parabolic regions} If $\bf Q$ is a proper $K$-parabolic subgroup of $\bf G$, then $\bf Q$ is conjugate over ${\bf G}(K)$ to ${\bf P}_I$ for some proper subset  $I \subsetneq \Delta$. We let $$\Lambda _{\bf Q} = \{\, \gamma f \in  {\bf G}(\mathcal{O}_S)F \mid \, ^{\gamma f}{\bf P}_I={\bf Q}\text{ for some }I \subsetneq \Delta\,\}$$ where $F$ is as in Theorem~\ref{t:finitecusps}. Notice that Theorem~\ref{t:finitecusps} insures that  $\Lambda _{\bf Q} $ is nonempty.

Given any $a=(a_v)_{v \in S} \in A$, and any $\alpha \in \Phi$, we let 
$$|\alpha (a)|=\prod_{v \in S} |\alpha (a_v)|_v$$ where $|\cdot |_v$ is the $v$-adic norm on $K_v$.

Given any $t>0$ and any $I \subsetneq \Delta$ we let $$A_I^+(t)=\{\, a \in A_I \mid |\alpha (a)| \geq t \text{ if } \alpha \in \Delta -I  \,\}$$ 
and we let $A_I^+=A_I^+(1)$.

For $t >0$, we let $$R_{\bf Q}(t) =\Lambda _{\bf Q} U_{\Phi(I)^+}{\bf M}_I(\mathcal{O}_S)A_I^+(t)$$ We call any such subset of $G$ a \emph{parabolic region}. We set $R_{\bf Q}=R_{\bf Q}(1)$.

\subsection{Boundaries of parabolic regions} We let $\partial A_I^+(t)$ be the set of all $ a \in A_I(t)$ such that there exists $ \alpha \in \Delta -I$ with  $|\alpha (a)| \leq  |\alpha (b)|$ for all  $b \in A_I(t)$.
Then we define the boundary of a parabolic region as
$$\partial R_{\bf Q}(t) =\Lambda _{\bf Q} U_{\Phi(I)^+}{\bf M}_I(\mathcal{O}_S) \partial A_I^+(t)$$

\subsection{Pruning $G$ to ${\mathbf{G}}(\mathcal{O}_S)$}  Given $0 \leq n < |\Delta|$, we let $\mathcal{P}(n)$ be the set of $K$-parabolic subgroups of $\bf G$ that are conjugate over ${\bf G}(K)$ to some ${\bf P}_I$ with $|I|=n$.

We will directly apply the following result from reduction theory to our proof of Theorem~\ref{t:p}.

\begin{proposition}\label{p:prune}
There exists a bounded set $B _0 \se G$, and given a bounded set $B_n\se G$ and any $N_n \geq 0$ for $0 \leq n <|\Delta|$, there exists $t_n >1$ and a bounded set $B_{n+1} \se G$  such that 
\begin{quote}
$(i)$ $G=\bigcup_{{\bf Q} \in \mathcal{P}(0)}R_{\bf Q}B_0$;
\medskip

\noindent $(ii)$ if ${\bf Q}, {\bf Q '} \in \mathcal{P}(n) $ and ${\bf Q} \neq {\bf Q'}$, then  the distance between $R_{\bf Q}(t_n)B_n$  and $R_{\bf Q'}(t_n)B_n$ is at least $N_n$;
\medskip

\noindent $(iii)$ ${\bf G}(\mathcal{O}_S) \cap R_{\bf Q}(t_n)B_n = \emptyset $ for all $n$;
\medskip

\noindent $(iv)$ if $n \leq |\Delta |-2$ then $\big(\bigcup_{{\bf Q} \in \mathcal{P}(n)} R_{\bf Q}B_n \big) - \big( \bigcup_{{\bf Q} \in \mathcal{P}(n)} R_{\bf Q}(2t_n)B_n \big)$ is contained in $  \bigcup_{{\bf Q} \in \mathcal{P}(n+1)}R_{\bf Q}B_{n+1}$;
\medskip

\noindent $(v)$ $\big(\bigcup_{{\bf Q} \in \mathcal{P}(|\Delta |-1)} R_{\bf Q}B_{|\Delta |-1} \big) - \big( \bigcup_{{\bf Q} \in \mathcal{P}(|\Delta |-1)} R_{\bf Q}(2t_{|\Delta |-1})B_{|\Delta |-1} \big)$ is contained in $  {\bf G}(\mathcal{O}_S)B_{ |\Delta |}$; and

\medskip

\noindent $(vi)$ if ${\bf Q} \in \mathcal{P}(n) $, then there is an $(L,C)$ quasi-isometry $R_{\bf Q}(t_n)B_n \rightarrow  U_{\Phi(I)^+}{\bf M}_I(\mathcal{O}_S)A_I^+$ for some $I \subsetneq \Delta$ with $|I|=n$. The quasi-isometry restricts to an $(L,C)$ quasi-isometry  $\partial R_{\bf Q}(t_n)B_n \rightarrow U_{\Phi(I)^+}{\bf M}_I(\mathcal{O}_S)\partial A_I^+$ where $L\geq 1$ and $C\geq 0$ are independent of $\bf Q$.
\medskip

\end{quote}
\end{proposition}

Proposition~\ref{p:prune} can be deduced from the Borel-Harish-Chandra-Behr-Harder reduction theory. The case when $K$ is a number field can be deduced from work of Borel (\cite{Boreltext} and \cite{Bo}) and the case when $K$ is a function field can be shown using Harder's work \cite{H}.

In the appendix, Section~\ref{s:a}, we give a more unified proof of Proposition~\ref{p:prune}.

\section{Filling spheres in solvable groups}\label{s:sol}

In this section we will prove

\begin{proposition}\label{p:solvable}
Let $I\se \Delta $, $0 < n \leq |S|-2$, and $r>0$.
There is some $m \in \mathbb{N}$ and $r'>0$ such that if $\Sigma$ is an $r$-coarse $n$-sphere in $U_{\Phi(I)^+}{\bf A}_I(\mathcal{O}_S)$, then there is an $r'$-coarse $(n+1)$-ball in $U_{\Phi(I)^+}{\bf A}_I(\mathcal{O}_S)$ whose volume is $O({\text{\emph{vol}}}(\Sigma )^m)$ and whose boundary is $\Sigma$.

\end{proposition}

Proposition~\ref{p:solvable} will be used to prove Proposition~\ref{p:bound} in the next section. Our proof of Proposition~\ref{p:solvable} is motivated by a proof of Gromov's that certain solvable Lie groups have a simply connected asymptotic cones (2.B.f  \cite{G}).

\subsection{Reducing to cells in slices} Before the next lemma, we need a couple of definitions.

Given a coarse manifold $\Sigma$ that is the image under a function $f$ of the vertices of a triangulated manifold $M$, a coarse (polysimplicial) subdivision of $\Sigma$ is an extension of $f$ to the vertices of a (polysimplicial) subdivision of $M$.

A $k$\emph{-slice} in $U_{\Phi(I)^+}{\bf A}_I(\mathcal{O}_S)$ is a left-coset of $$\Big[\prod_{v\in S'}{\bf U}_{\Phi(I)^+}(K_v)\Big]{\bf A}_I(\mathcal{O}_S)$$ for some $S'\se S$ with $|S'|=k$.

\begin{lemma}\label{l:sub}
Given $r>0$ and $n\in \mathbb{N}$, there is some $r'>0$ such that any coarse $n$-sphere $\Sigma \se U_{\Phi(I)^+}{\bf A}_I(\mathcal{O}_S)$ of scale $r$ can be subdivided into a coarse polysimplicial $n$-sphere $\Sigma '$ of scale $r'$ such that every coarse $k$-cell in $\Sigma '$ is contained in a $k$-slice, and such that $\text{\emph{vol}}(\Sigma ')= O(\text{\emph{vol}}(\Sigma ))$.
\end{lemma}

\begin{proof}

Let $$\pi _v :U_{\Phi(I)^+}{\bf A}_I(\mathcal{O}_S) \rightarrow {\mathbf{U}}_{\Phi(I)^+}(K_v)$$ and  $$\pi _A :U_{\Phi(I)^+}{\bf A}_I(\mathcal{O}_S) \rightarrow {\bf A}_I(\mathcal{O}_S)$$ be the obvious projection maps.

Let $\sigma \se \Sigma$ be a $k$-simplex and  let $x_1,x_2,...,x_{k+1}$ be its vertices. 
For each $v \in S$, we form an abstract $k$-simplex with  vertices  $\pi _v (x_1),...,\pi _v (x_{k+1})$. Call this simplex $Z_{\sigma,v}$. Let $Z_{\sigma ,A}$ be the abstract $k$-simplex  with  vertices  $\pi _A (x_1),...,\pi _A (x_{k+1})$.

Let $$Z_\sigma=\big(\prod _{v \in S} Z_{\sigma,v}\big)\times Z_{\sigma,A}$$ 
Notice that $Z_\sigma$  is a polysimplicial complex, that is homeomorphic to a $(|S|+1)k$-ball, and that the number of cells in $Z_\sigma$ is bounded by a constant depending only on $k$ and $|S|$.
There is also an obvious function from the vertices of $Z_\sigma$ into $U_{\Phi(I)^+}{\bf A}_I(\mathcal{O}_S)$, which we will denote as $h$.

Because $\Sigma$ has scale $r$, there is some $r'>0$ depending only on $r$ and $n$ such that if $u$ and $w$ are vertices in $Z_\sigma$, then the distance between $h(u)$ and $h(w)$ is at most $r'$. Thus, the vertices of any $k$-cell in $Z_\sigma$ maps via $h$ to a coarse $k$-cell in $U_{\Phi(I)^+}{\bf A}_I(\mathcal{O}_S)$ of diameter at most $r'$. Furthermore, any such coarse $k$-cell must be contained in a $k$-slice, since it projects to a positive dimensional simplex in at most $k$ of the $Z_{\sigma,v}$ factors.

Vertices of $\Sigma$ are clearly contained in 0-slices. Suppose $\sigma \se \Sigma$ is a coarse 1-simplex. Its faces --- that is, its endpoints --- are represented by vertices in $Z_\sigma$. Connect them with a path  $\widetilde{\sigma}$ in the 1-skeleton of $ Z_\sigma$.  Since  the number of edges in $Z_\sigma$ is bounded, $\widetilde{\sigma}$ consists of a bounded number of edges. Map the vertices of  $\widetilde{\sigma}$ into 
$U_{\Phi(I)^+}{\bf A}_I(\mathcal{O}_S) $
 via $h$, and we have subdivided $\sigma$ into a uniformly bounded number of coarse 1-polysimplices of scale $r'$ such that each coarse $1$-polysimplex is contained in a $1$-slice. 

We continue by induction. 

Let $\sigma$ now be a coarse $k$-simplex in $\Sigma$. By induction hypothesis we may assume that the faces of $\sigma$, named $\tau_1,\tau_2,...,\tau_{k+1}$, have been subdivided into coarse $(k-1)$-polysimplices of scale $r'$ that are contained in $(k-1)$-slices.

The subdivided $\tau _i$ are represented by complexes  $\widetilde{\tau_i} \se Z_{\tau_i}$. Since $\tau_i \se \sigma$, we have $Z_{\tau_i} \se Z_\sigma$ and that $\cup _i \widetilde{\tau_i} \se Z_\sigma$ is the continuous image of a $(k-1)$-sphere.

 Because $Z_\sigma$ has a uniformly bounded number of polysimplices, there is a topological $k$-ball $\widetilde{\sigma} \se Z_\sigma$ whose boundary is  $\cup _i \widetilde{\tau_i} \se Z_\sigma$ and such that via $h$, the vertices of $\widetilde{\sigma}$ represent a coarse polysimplicial $k$-ball of scale $r'$ in $U_{\Phi(I)^+}{\bf A}_I(\mathcal{O}_S) $, whose volume is bounded by a constant depending on $n$ and $|S|$, and such that each coarse $k$-polysimplex in the coarse ball is contained in a $k$-slice.
 
 The result is a coarse polysimplicial subdivision of $\Sigma$ all of whose coarse $k$-polysimplices are contained in $k$-slices.

 \end{proof}

\subsection{The geometry of slices}

The convenience of reducing the problem of filling spheres in $ U_{\Phi(I)^+}{\bf A}_I(\mathcal{O}_S)$ to filling in spheres in $k$-slices of  $U_{\Phi(I)^+}{\bf A}_I(\mathcal{O}_S)$ is that $k$-slices are negatively curved in some sense. Our proof of Lemma~\ref{l:slicefill} makes these remarks more precise, but first we will need

\begin{lemma}\label{l:tproj}
If $I \se \Delta$ and $S'$ is a proper subset of $S$, then the projection of ${\bf A}_I(\mathcal{O}_S)$ into $\prod_ {v \in S'} {\bf A}_I(K_v)$ is a finite Hausdorff distance from $\prod_ {v \in S'} {\bf A}_I(K_v)$.
\end{lemma}

\begin{proof}
It suffices to prove the lemma for the case when $S' =S-\{w\}$ for some $w \in S$. 

Notice that the geometric dimension of ${\bf A}_I(\mathcal{O}_S)$ equals $(|S|-1)(|\Delta -I|)$ (by Dirichlet's units theorem), as does  $\prod_ {v \in S'} {\bf A}_I(K_v)$. So it suffices to check that the kernel of the projection ${\bf A}_I(\mathcal{O}_S)\rightarrow \prod_ {v \in S'} {\bf A}_I(K_v)$ is bounded. 
But if $a \in {\bf A}_I(\mathcal{O}_S)$ is trivial in  ${\bf A}_I(K_v)$ for each $v \in S'$, then by the product formula, $a$ has trivial norm in  ${\bf A}_I(K_w)$, and thus $a$ is bounded in $A_I$.

\end{proof}

\begin{lemma}\label{l:slicefill}

Let $I \se \Delta$, let $k<|S|$, and let $n,r \geq 0$. Then there is some $r'>0$ such that any coarse polysimplicial $n$-sphere $\Sigma$ of scale $r$ that is contained in a $k$-slice of  $U_{\Phi(I)^+}{\bf A}_I(\mathcal{O}_S)$
 bounds a coarse polysimplicial  $(n+1)$-ball, denoted $D$, which is of scale $r'$ and is contained in the same $k$-slice with ${\text{\emph{vol}}}(D)=O({\text{\emph{diam}}}(\Sigma)^{n+1}+\text{\emph{vol}}(\Sigma)\text{\emph{diam}}(\Sigma))$.

\end{lemma}
  
  \begin{proof}
  After left translation, we may assume $$1\in \Sigma \se \prod_{v \in S'} {\bf U}_{\Phi(I)^+}(K_v){\bf A}_I(\mathcal{O}_S)$$ where $|S'|=k$.
  
   Let $L$ be the diameter of $\Sigma$, and choose $u \in \prod_{v \in S'} {\bf U}_{\Phi(I)^+}(K_v)$ and $a \in {\bf A}_I(\mathcal{O}_S)$ such that $u a \in \Sigma$. Then $d(1,a)\leq L$.

By previous lemma, there is some $b \in {\bf A}_I(\mathcal{O}_S)$
such that $|\alpha(b)|_v <1$ for all $\alpha \in \Delta -I$ and $v \in S'$. Therefore,  $|\beta(b)|_v <1$ for all $\beta \in \Phi(I)^+$ and $v \in S'$. 

Notice that for $N=O(L)$, $d(ab^N,uab^N)=d(a,(b^{-N}ub^N)a) \leq 1$. Thus, we may assume that 
$  \Sigma b^N$ is contained in $ {\bf A}_I(\mathcal{O}_S)$, which is quasi-isometric to Euclidean space and thus there is an $r$-coarse polysimplicial ball $D' \se {\bf A}_I(\mathcal{O}_S)$ whose volume is $O(L^{n+1})$ and whose boundary is $  \Sigma b^N$.
Therefore, we can let $$D=D' \cup \bigcup_{i=0}^N  \Sigma b^N$$
  
  \end{proof}
  
 Notice that the volume of a coarse $0$ sphere equals $2$, and if $n>0$, then the diameter of an $r$-coarse $n$-sphere is bounded by $r(\text{vol}(\Sigma))$. Thus from Lemma~\ref{l:slicefill} we have the following two corollaries.

\begin{corollary}\label{c:path}

Let $I \se \Delta$, let $k<|S|$. Then there is some $r'>0$ such that if $x$ and $y$ are any two points in a $k$-slice of  $U_{\Phi(I)^+}{\bf A}_I(\mathcal{O}_S)$, then $x$ and $y$ are the endpoints of an $r'$-coarse path whose volume (or length) is $O(d(x,y))$. 

\end{corollary}
  
\begin{corollary}\label{c:sphere}

Let $I \se \Delta$, let $k<|S|$, and let $n,r > 0$. Then there is some $r'>0$ such that any coarse polysimplicial $n$-sphere of scale $r$ contained in a $k$-slice of  $U_{\Phi(I)^+}{\bf A}_I(\mathcal{O}_S)$
 bounds a coarse polysimplicial $(n+1)$-ball of scale $r'$ in the same $k$-slice with ${\text{\emph{vol}}}(D)=O({\text{\emph{vol}}}(\Sigma)^{n+1})$.

\end{corollary}

\subsection{Filling spheres that are piecewise in slices}  We use the negative curvature of $k$-slices to prove

\begin{lemma}\label{l:foo} For $r>0$ and $n \leq |S|-2$, there is some $m \in \mathbb{N}$ and  $r'>0$ such that
if $\Sigma \se U_{\Phi(I)^+}{\bf A}_I(\mathcal{O}_S)$ is an $r$-coarse polysimplicial $n$-sphere all of whose coarse $k$-cells are contained in $k$-slices, then there is
 an $r'$-coarse polysimplicial $(n+1)$-ball $D$ such that $\partial D = \Sigma$ and $\text{\emph{vol}}(D)=O(\text{\emph{vol}}(\Sigma)^m)$.
\end{lemma}

\begin{proof}

We may assume that $1 \in \Sigma$.

Fix $w \in S$ and let  $$\pi : U_{\Phi(I)^+}{\bf A}_I(\mathcal{O}_S) \rightarrow \Big[\prod_{v \in (S-w)}{\bf U}_{\Phi(I)^+}(K_v)\Big]{\bf A}_I(\mathcal{O}_S)$$ be the obvious projection onto the $(|S|-1)$-slice.

Let $$D' \se  \prod_{v \in (S-w)}{\bf U}_{\Phi(I)^+}(K_v){\bf A}_I(\mathcal{O}_S)$$ be the coarse polysimplicial $(n+1)$-ball  with $\partial D' =\pi(\Sigma)$ that is given by Corollary~\ref{c:sphere}.

If $x$ is a vertex in $\Sigma$, then by Corollary~\ref{c:path} we can connect $x$ with $\pi (x)$ with the coarse path of length $O(\text{vol}(\Sigma ))$ contained in the $1$-slice 
that is the coset of ${\bf U}_{\Phi(I)^+}(K_w){\bf A}_I(\mathcal{O}_S)$ containing $x$ and $\pi (x)$. Call this path $p_x$.

Given a coarse $1$-cell  $\sigma \se \Sigma$ with endpoints $x_1$ and $x_2$, notice that $$\sigma \cup \pi (\sigma) \cup p_{x_2} \cup p_{x_1}$$ is a coarse loop contained in a $2$-slice, since $\sigma$ is contained in a $1$-slice. Therefore, by Corollary~\ref{c:sphere} there is a coarse 
 polysimplicial 2-disk $D_\sigma$ contained in the $2$-slice whose volume is $O(\text{vol}(\Sigma)^2)$ and whose boundary is the loop above.

We continue by induction. If $\sigma \se \Sigma$ is a coarse $k$-cell with faces $\tau_1,...,\tau_m$, then 
$$\sigma \cup \pi (\sigma) \cup (\cup _{i=1}^m D_{\tau_i})$$ is a coarse polysimplicial $k$-sphere contained in a $(k+1)$-slice. Since $k \leq n$, we have $k+1 \leq  |S|-1$, and thus by Corollary~\ref{c:sphere}, there is a coarse polysimplicial $(k+1)$-disk  $D_\sigma$ contained in the $(k+1)$-slice whose volume is polynomial in $\text{vol}(\Sigma)$, and whose boundary is the above coarse polysimplicial $k$-sphere.

Let $\mathcal{C}$ be the set of all $n$-cells in $\Sigma $. Then 
$$D=D' \cup (\cup _{\sigma \in \mathcal{C}}D_\sigma)$$
satisfies the lemma. \end{proof}

\subsection{Proof of Proposition~\ref{p:solvable}} Subdivide $\Sigma$ to $\Sigma '$ using Lemma~\ref{l:sub}. Then $\Sigma '$ bounds a coarse polysimplicial ball $D$ by Lemma~\ref{l:foo}.

Since the Hausdorff distance between $\Sigma$ and $\Sigma '$ is bounded, we may use $D$ and the quasi-isometry from $G$ to its associated product of symmetric spaces and Euclidean buildings to realize the desired coarse $(n+1)$-ball in ${ U}_{\Phi (I)^+}{\bf A}_{I}(\mathcal{O}_S)$.

\subsection{Proof of Theorem~\ref{p:s}} The proposition follows immediately from Proposition~\ref{p:solvable}. Indeed, we may assume that ${\bf Q}={\bf P}_I$ for some proper subset $I \subsetneq \Delta$. Thus, ${\bf U_Q}={\bf U}_{\Phi (I)^+}$, any $K$-Levi subgroup of ${\bf P}_I$ is conjugate over ${\bf P}_I(K)$ to ${\bf M}_I{\bf A}_I$, and the maximal $K$-split torus of the center of ${\bf M}_I{\bf A}_I$ is ${\bf A}_I$. Furthermore, since ${\bf U}_{\Phi (I)^+}$ is unipotent, ${\bf U}_{\Phi (I)^+}(\mathcal{O}_S)$ is a cocompact lattice in, and thus is quasi-isometric to, ${ U}_{\Phi (I)^+}$.

\section{Filling manifolds in the boundaries of parabolic regions}\label{s:bound}

For a proper subset $I \subsetneq \Delta $, we let $A_I^+=A_I^+(1)$, $R_I=U_{\Phi(I)^+}{\bf M}_I(\mathcal{O}_S)A_I^+$, and $\partial R_I=U_{\Phi(I)^+}{\bf M}_I(\mathcal{O}_S)\partial A_I^+$.

In this section we use Proposition~\ref{p:solvable} to prove that a coarse manifold in the  boundary of a parabolic region has a polynomially efficient filling in the same boundary. The precise statement is given as

\begin{proposition}\label{p:bound}
There is some $m \in \mathbb{N}$, and given $r>0$, there is some $r'>0$  such that the following holds:

If  $n \leq |S|-1$ and $\Sigma \se R_I$ is an  $r$-coarse $n$-manifold whose volume and maximum distance from $1$ is bounded by $d>0$, and whose boundary components are contained in $ \partial R_I$, then there is  an $r'$-coarse $n$-manifold $\Sigma'\se \partial R_I$ of the same topological type as $\Sigma$, with $\partial \Sigma' =\partial \Sigma$ and $$\text{\emph{vol}}(\Sigma')= O(d ^m)$$
\end{proposition}

\begin{proof} 

The quotient map $$ R_I \rightarrow {\bf{M}}_I(\mathcal{O}_S)  A_I^+ $$ is distance nonincreasing.

Choose a Lipschitz map $ A_I^+ \rightarrow \partial A_I^+$ that is the identity on $\partial A_I^+$. Because ${\bf M}_I$ and $ {\bf A}_I$ commute, ${\bf{M}}_I(\mathcal{O}_S)A_I^+$ is a metric direct product, so the induced map 
$$ {\bf{M}}_I(\mathcal{O}_S)  A_I^+ \rightarrow  {\bf{M}}_I(\mathcal{O}_S) \partial A_I^+$$ is Lipschitz.

We define $q_I$ to be the composition of maps $$ R_I \rightarrow {\bf{M}}_I(\mathcal{O}_S)  A_I^+ \rightarrow {\bf{M}}_I(\mathcal{O}_S) \partial A_I^+$$ so that $q_I$ is Lipschitz, and takes values in $\partial R_I$.

Suppose $x_i \in \partial \Sigma $. Then $x_i=u_im_ia_i$ for some $u_i \in U_{\Phi(I)^+}$, $m_i \in {\bf{M}}_I(\mathcal{O}_S)$, and $a_i  \in \partial A_I^+$.

Given $i$, choose some $w \in S$, and let $u_w \in {\bf U}_{\Phi(I)^+}(K_w)$ and $u_{-w} \in \prod_{v\in S-w}{\bf U}_{\Phi(I)^+}(K_v)$ be such that $u_wu_{-w}= (m_ia_i)^{-1}u_im_ia_i$. 

By Corollary~\ref{c:path}, there is a course path $f_w \se {\bf U}_{\Phi(I)^+}(K_w) {\bf A}_I(\mathcal{O}_S)$ of length $O(d)$ whose endpoints are 1 and $u_w$. By the same corollary,  there is a coarse path $f_{-w} \se \big[ \prod_{v\in S-w}{\bf U}_{\Phi(I)^+}(K_v)\big] {\bf A}_I(\mathcal{O}_S)$ of length $O(d)$ whose endpoints are 1 and $u_{-w}$. We let  $f_i \se U_{\Phi(I)^+} \, {\bf A}_I(\mathcal{O}_S) $ be the union of $f_w$ and $u_{w}f_{-w}$, so that $f_i$ is a coarse path from $1$ to $(m_ia_i)^{-1}u_im_ia_i$ whose length is $O(d)$.

 Thus, 
\begin{align*} m_ia_if_i  & \se m_ia_i U_{\Phi(I)^+} \, {\bf A}_I(\mathcal{O}_S) \\
& \se U_{\Phi(I)^+} \, m_i a_i\, {\bf A}_I(\mathcal{O}_S)  \\
& \se U_{\Phi(I)^+} \, m_i \, \partial A_I^+\\
&  \se     U_{\Phi(I)^+} \,  {\bf M}_I(\mathcal{O}_S) \, \partial A_I^+ \\
& \se   \partial R_I  \end{align*}
 is a path of length $O(d)$ that connects $m_ia_i=q_I(x_i)$ to $u_im_ia_i=x_i$. We name this path $D(x_i)$.

\bigskip

In what remains of this proof, we will denote a coarse $k$-simplex in $\partial \Sigma$ by the $(k+1)$-tuple of its vertices.

\bigskip

We claim that for any coarse simplex $(x_1,...,x_k)$ in $\partial \Sigma$, there is a coarse $k$-disk $D  (x_1,...,x_k) $ such that 

\begin{quote}

\bigskip \noindent ($i$) $D  (x_1,...,x_k) \se m_{i} a_{i} U_{\Phi(I)^+} \, {\bf A}_I(\mathcal{O}_S) B_k$ for any $1\leq i \leq k$ and some compact set  $B_k \se G$ of radius depending on $k$ with $1 \in B_k$.

\bigskip \noindent ($ii$) $\partial D  (x_1,...,x_k)$ is the union of $ (x_1,...,x_k)$,  $q_I (x_1,...,x_k)$, and $ \cup _{i=1}^kD  (x_1,..., \widehat{x_i} ,...,x_k)$ where $(x_1,..., \widehat{x_i} ,...,x_k)$ denotes the simplex obtained by removing the vertex $x_i$ from the simplex $(x_1,...,x_k)$.

\bigskip \noindent ($iii$) The volume of $D(x_1,...,x_k)$ is $O(d^m)$ for some $m$ depending on $k$.

\end{quote}

We prove our claim by induction on $k$. The case when $k=1$ is resolved, so we assume our claim is true for $k$ and assume that $(x_1,...,x_{k+1})$ is a coarse simplex in $ \partial \Sigma$.

Let $s(x_1,\ldots , x_{k+1})$ be the union of $(x_1,...,x_{k+1})$,  $q_I(x_1,...,x_{k+1})$, and $    \cup _{i=1}^{k+1}D  (x_1,..., \widehat{x_i} ,...,x_{k+1})$. By the induction hypothesis,
$s(x_1,\ldots , x_{k+1})$ is a coarse sphere of dimension $k \leq |S|-2$.

If $B \se G$ is the ball of radius $r$ around $1$, and $1 \leq i,j \leq k+1$, then $(m_ia_i)^{-1}m_ja_j \in B$ by our assumption on the scale of $\Sigma$.

Recall that $m_i, m_j \in M_I$, so they commute with ${\bf A}_I(\mathcal{O}_S)$ and normalize $U_{\Phi(I)^+}$, as do $a_i$ and $a_j$. Thus,
\begin{align*}
  m_ja_j U_{\Phi(I)^+} \, {\bf A}_I(\mathcal{O}_S) B_k
& = m_ia_i U_{\Phi(I)^+} \, {\bf A}_I(\mathcal{O}_S) (m_ia_i)^{-1}m_ja_j B_k\\
& \se m_ia_i U_{\Phi(I)^+} \, {\bf A}_I(\mathcal{O}_S) B B_k
\end{align*}
Therefore, we let $B_{k+1}=BB_{k}$ so that $s(x_1,\ldots , x_{k+1})$ is contained in $m_{i} a_{i} U_{\Phi(I)^+} \, {\bf A}_I(\mathcal{O}_S) B_{k+1}$ for any $1\leq i \leq k+1$.

By Proposition~\ref{p:solvable}, there is a coarse $(k+1)$-ball in $m_{1} a_{1} U_{\Phi(I)^+} \, {\bf A}_I(\mathcal{O}_S) B_{k+1}$ whose boundary is $s(x_1,\ldots , x_{k+1})$, and whose volume is bounded by a polynomial in the volume of $s(x_1,\ldots , x_{k+1})$. We name this ball $ D  (x_1 ,...,x_{k+1})$. This justifies our claim.

\bigskip

Now let $X$ be the union of $D (x_1,...,x_{n} )$ taken over maximal simplices $ (x_1,...,x_{n} )$ in  $ \partial \Sigma$. Then $X$ is a coarse $n$-manifold in 
\begin{align*}{\bf M}_I(\mathcal{O}_S) (\partial A_I^+) U_{\Phi(I)^+} \, {\bf A}_I(\mathcal{O}_S)B_{n} \, & \se \, 
 U_{\Phi(I)^+} \, {\bf M}_I(\mathcal{O}_S)  \partial A_I^+B_{n} \, \\ & \se \, \partial R_I B_{n}\end{align*}

The volume of $X$ is $O(d^m)$ for some $m$, and the boundary of $X$ is $\partial \Sigma \cup q_I(\partial \Sigma)$. Notice that $X$ establishes something like a ``polynomial homotopy" between $q_I$ restricted to $\partial \Sigma$ and the identity map restricted to $\partial \Sigma$.

 We let $\Sigma '$ be the image of  $q_I(\Sigma) \cup X$ under the obvious rough isometry $R_I B_{n} \rightarrow R_I$.

\end{proof}

\section{Proof of Main Result (Theorem~\ref{t:p})}\label{s:proof}

Let $B_0$ be as in Proposition~\ref{p:prune} and 
let $\Sigma \se G =  \bigcup _{{\bf Q} \in \mathcal{P}(0)}R_{\bf Q}B_0  $ be a coarse $n$-manifold of scale $r_0$ with $\partial \Sigma \se {\bf G}(\mathcal{O}_S)$ and $n < |S|$. 

We relabel $\Sigma$ as $\Sigma _0$, and we let $N_0=2r_0$. Then let $B_1$ and $t_0$ be as in Proposition~\ref{p:prune}.

For ${\bf Q} \in \mathcal{P}(0)$, we define the coarse manifolds
$$\Sigma _{0, {\bf Q}} =  \Sigma _0 \cap  R_{\bf Q}(t_0)B_0 $$
and 
 $$ \Sigma _{0, \partial} =  \Sigma _0 - \bigcup _{{\bf Q} \in \mathcal{P}(0)}\Sigma _{0,{\bf Q}}$$ 

Notice that $\Sigma _{0, {\bf Q}}  \cap \Sigma _{0, {\bf Q'}} =\emptyset$ if ${\bf Q} \neq {\bf Q'}$ by part $(ii)$ of Proposition~\ref{p:prune}.

For each ${\bf Q} \in \mathcal{P}(0)$ we may perturb the points in $\Sigma _{0, {\bf Q}}$ by distance at most $r_0$ such that $\partial \Sigma _{0, {\bf Q}}  \se \partial R_{\bf Q}(t_0)B_0$, the latter set being quasi-isometric (with constants independent of  $\bf Q$) to $\partial R_\emptyset$ by Proposition~\ref{p:prune}$(vi)$.

By Proposition~\ref{p:bound}, there is some $r_1 > 0$ (that depends  on the constants of the above quasi-isometry) and  a coarse manifold $\Sigma _{0, {\bf Q}}' \se  \partial R_{\bf Q}(t_0)B_0$ of scale $r_1$ for each ${\bf Q} \in \mathcal{P}(0)$ such that 
the coarse manifold \begin{align*} \Sigma _1 & =  \Sigma _{0, \partial} \cup  \bigcup _{{\bf Q} \in \mathcal{P}(0)}\Sigma _{0,{\bf Q}}' \\ \end{align*}
is a coarse manifold of scale $r_1$, of the same topological type as $\Sigma _0$, and whose volume is $O(\text{vol}(\Sigma)^k)$ for some $k$. 

Also note that by Proposition~\ref{p:prune}$(iv)$
 \begin{align*} \Sigma _1 & \, \se \,  \big( \bigcup _{{\bf Q} \in \mathcal{P}(0)} R_{\bf Q}B_0 \big) - \big( \bigcup _{{\bf Q} \in \mathcal{P}(0)} R_{\bf Q}(2t_0)B_0 \big) \\
& \, \se \,  \bigcup _{{\bf Q} \in \mathcal{P}(1)}R_{\bf Q}B_1 \end{align*}

Furthermore, since ${\bf G}(\mathcal{O}_S) \cap R_{\bf Q}(t_0)B_0 = \emptyset $ by Proposition~\ref{p:prune}$(iii)$, we have that $\partial \Sigma _0 = \partial \Sigma_1$.

Repeat this argument with $1\leq n < |\Delta|-1$ in the place of $0$ above. The result is a coarse manifold $ \Sigma _ {|\Delta|}$ that is of scale $r_{|\Delta|}>0$ (for some $r_{|\Delta|}>0$); that is of the same topological type as $\Sigma _0$; with $\partial \Sigma_{|\Delta|} = \partial \Sigma_0$; whose volume is $O(\text{vol}(\Sigma)^k)$ for some $k$;
 and that is contained in  $$  
 \big( \bigcup _{{\bf Q} \in \mathcal{P}( {|\Delta|-1})} R_{\bf Q}B_ {|\Delta|-1} \big) - \big( \bigcup _{{\bf Q} \in \mathcal{P}( {|\Delta|-1})} R_{\bf Q}(2t_ {|\Delta|-1})B_ {|\Delta|-1} \big) $$ and hence, is contained in $ {\bf G}(\mathcal{O}_S)B_ {|\Delta|}$ by Proposition~\ref{p:prune}$(v)$.

As ${\bf G}(\mathcal{O}_S)B_ {|\Delta|}$ is a finite Hausdorff distance from ${\bf G}(\mathcal{O}_S)$, our proof is complete.

\section{Isoperimetric inequalities}\label{s:finisht}
In this section, we prove that our main result implies Corollary~\ref{c:isop}. That is, we show that 
 ${\bf G} (\mathcal{O}_S)$  satisfies a polynomial $m$-dimensional isoperimetric inequality when $m \leq |S|-2$.

Let $X$ be an $(|S|-2)$-connected CW-complex that ${\bf G} (\mathcal{O}_S)$ acts on, cellularly, freely, properly, and cocompactly. Choose a basepoint $x \in X$ and let $\phi :{\bf G} (\mathcal{O}_S) \rightarrow {\bf G} (\mathcal{O}_S)\cdot x$ be the orbit map. It is a bijective quasi-isometry where ${\bf G} (\mathcal{O}_S)$ is endowed with the restriction of the left-invariant metric on ${\bf G} (\mathcal{O}_S)$, and  ${\bf G} (\mathcal{O}_S)\cdot x$ is endowed with the restriction of the path metric on $X$.

Let $\Sigma \se X$ be a cellular $m$-sphere for $m\leq |S|-2$. Every point in $\Sigma$ is a uniformly bounded distance from a point in  the orbit ${\bf G} (\mathcal{O}_S)\cdot x$. Thus, there exists some $r_0>0$ such that  after perturbing $\Sigma$ by a uniformly bounded amount, the Hausdorff distance between $\Sigma$ and $\Sigma \cap {\bf G} (\mathcal{O}_S)\cdot x$ is uniformly bounded and $\Sigma \cap {\bf G} (\mathcal{O}_S)\cdot x$ is an $r_0$-coarse $m$-sphere.

Therefore, $\phi^{-1}(\Sigma \cap {\bf G} (\mathcal{O}_S)\cdot x)$ is an $r_1$-coarse $m$-sphere in ${\bf G} (\mathcal{O}_S)$ for some $r_1>0$ that depends only on $r_0$ and the quasi-isometry constants of $\phi$.

Since $G$ is quasi-isometric to a CAT(0) space, there is an $r_1$-coarse $(m+1)$-disk $D \se G$ with $\partial D = \phi^{-1}(\Sigma \cap {\bf G} (\mathcal{O}_S)\cdot x)$ and with $\text{vol}(D)=O(\text{vol}(\Sigma)^{\frac{m+1}{m}})$.

By Theorem~\ref{t:p}, there is some $r_2>0$ and a polynomial $f$ such that there exists an $r_2$-coarse $(m+1)$-disk $D' \se {\bf G} (\mathcal{O}_S)$ with $\partial D' = \phi^{-1}(\Sigma \cap {\bf G} (\mathcal{O}_S)\cdot x)$ and  $\text{vol}(D')=f(\text{vol}(\Sigma)^{\frac{m+1}{m}})$.

There is some $r_3>0$ depending only on $r_2$ and the quasi-isometry constants of $\phi$ such that $\phi(D') \se X$ is an $r_3$-coarse $(m+1)$-disk with boundary $\Sigma \cap {\bf G} (\mathcal{O}_S)\cdot x$ and $\text{vol}(\phi(D'))=\text{vol}(D')$.

Starting with the 0-skeleton given by $\phi(D')$, we connect adjacent vertices in the coarse manifold $\phi(D')$ with 1-cells. If the two adjacent points are contained in $\Sigma \cap {\bf G} (\mathcal{O}_S)\cdot x$, then we use the 1-cell that connects them in $\Sigma$. We continue by the dimension of the skeleton to define a 
 topological $(m+1)$-ball $D'' \se X$ that is a uniformly bounded Hausdorff distance from $\phi(D')$, whose boundary is $\Sigma$, and that contains $O(\text{vol}(D'))$ many cells. This proves Corollary~\ref{c:isop}.

Notice that if $f$ were a linear polynomial, the proof above establishes a Euclidean isoperimetric inequality, and thus that Conjecture~\ref{c:p} implies Conjecture~\ref{c:iso}.

\section{Appendix: Reduction Theory and a Proof of Proposition~\ref{p:prune}}\label{s:a}

In this Section we provide a proof of Proposition~\ref{p:prune}. 

Throughout this section, $F$ will be the set given in Theorem~\ref{t:finitecusps}. We will also make use of the notation introduced in Section~\ref{s:p}.

\bigskip

We begin with the main result from reduction theory:

\begin{theorem}\label{t:red} There is a bounded set $B \se G$  such that $${\bf G}(\mathcal{O}_S) F U M A_\emptyset ^+ B =G$$
\end{theorem}

\begin{proof}

Springer 2.1F \cite{S} provides an adelic version of this theorem. See Godement Theorem 11 \cite{Go} for the proof that the adelic version implies this theorem using Theorem~\ref{t:finitecusps}. (Theorem 11 in \cite{Go} is stated for number fields, but the proof works for an arbitrary global field.) 

\end{proof}

\subsection{Root choice} We will need to make use of a carefully chosen positive root. We explain that choice below along with some related notation.

Given $w \in W$ and $\alpha \in \Phi$, we let $\alpha ^w \in \Phi$ be defined by $\alpha ^w(a)=\alpha (w^{-1}aw)$. Then we define $\Delta ^w=\{\,\alpha ^w \mid \alpha \in \Delta \,\}$, and we let  $(\Phi ^w)^+$ and $(\Phi ^w)^-$ be the set of positive and negative roots respectively with respect to the simple roots $\Delta ^w$.

The sets of roots $\Phi^+$, $(\Phi^w)^+$, and $(\Phi^w)^-$ are closed under addition, so $\Phi^+\cap (\Phi^w)^+$ and $\Phi^+\cap (\Phi^w)^-$ are as well. Hence, there are corresponding unipotent subgroups ${\bf {U}}_{\Phi^+\cap (\Phi^w)^+}, {\bf {U}}_{\Phi^+\cap (\Phi^w)^-} \leq {\bf U}_{\Phi^+}$ that we label as ${ \bf U}_{w,+}$ and ${ \bf U}_{w,-}$ respectively. 
They are each normalized by $\bf MA \leq Z_G(A)$ (see e.g. 21.9(ii) \cite{Boreltext}).

In the case that $w \neq 1$, we have that ${\Phi^+\cap (\Phi^w)^-} \neq \emptyset$, and the next lemma determines a choice of root $\tau _{J,I,w} \in {\Phi^+\cap (\Phi^w)^-}$ that satisfies some properties we'll need to make use of later in this section. As an example of the following lemma, if $w$ represents the longest element of the Weyl group then ${\Phi^+\cap (\Phi^w)^-}={\Phi^+}$ and $\tau _{J,I,w}$ is the highest root with respect to $\Phi ^+$.

\begin{lemma}\label{l:root1}
Suppose $I,J \subsetneq \Delta$, that $w \in W$, and $w \notin {\bf P}_J(K)$. Then there is some $\tau _{J,I,w} \in {\Phi^+\cap (\Phi^w)^-}$ such that
\begin{quote}
$(i)$  if $\alpha \in {\Phi^+\cap (\Phi^w)^-}$ then $\alpha + \tau _{J,I,w} \notin \Phi$;
\medskip 

\noindent $(ii)$   if \mbox{$a \in \,A^+_I(1)$} then $ |\tau _{J,I,w}(a)|\geq 1$; and 
\medskip 

\noindent $(iii)$ for any $r >0$ there is some $t >1$ such that \mbox{$a \in \,^wA^+_J(t)$} implies $ |\tau _{J,I,w}(a)|<r$.
\end{quote}
\end{lemma}

\begin{proof}
Let $\Sigma$ be the apartment of the spherical building for ${\bf G}(K)$ that corresponds to $\bf A$. Let $\mathfrak{C}_{\bf P}, \mathfrak{C}_{^w\bf P},  \mathfrak{C}_{^w{\bf P}_J} \se \Sigma$ be the simplices corresponding to $\bf P$, $^w \bf P$, and  $^w{\bf P}_J$ respectively.

We write $\Delta $ as $\{\, \alpha _1,\ldots , \alpha _n\,\}$, and for each $j$, let $\mathcal{H}_j \se \Sigma$ be the simplicial hemisphere corresponding to $\alpha _j$. 

Either 
$\mathfrak{C}_{^w{\bf P}_J} \nsubseteq \mathcal{H}_j$ for some $j$, or else $\mathfrak{C}_{^w{\bf P}_J} \subseteq \cap_{j=1}^n\mathcal{H}_j= \mathfrak{C}_{{\bf P}} $ which implies through the type preserving action of the Weyl group on $\Sigma$ that $^w {\bf P}_J = {\bf P}_J $ and thus that $w \in {\bf P}_J(K)$. 

By the hypotheses of the lemma, we proceed under the assumption that there is some $j$ such that $\mathfrak{C}_{^w{\bf P}_J} \nsubseteq \mathcal{H}_j$.
Then clearly $\mathfrak{C}_{^w{\bf P}} \nsubseteq \mathcal{H}_j$, so that $\alpha _j \in  {\Phi^+\cap (\Phi^w)^-}$.

Let $\Delta ^w=\{\, \beta _1,\ldots , \beta _n\,\}$ and choose the ordering on the roots such that $J ^w=\{\, \beta _1,\ldots , \beta _{|J|}\,\}$. Because $\alpha_j \in (\Phi^w)^- $, we have $\alpha_j= \sum _i m_{0,i} \beta _i$ where $m_{0,i} \leq 0$ for all $i$.

Choose a root $\tau _1  \in \Phi^+ \cap  (\Phi^w)^- $ such that $\tau_1 = \sum _i m_{1,i} \beta _i$ where $m_{1,i} \leq m_{0,i}$ for all $i$ and such that of all possible choices for $\tau _1$ as above, the coefficient $m_{1,1}$ is minimal.

Then choose a root $\tau _2  \in \Phi^+ \cap  (\Phi^w)^- $ such that $\tau_2 = \sum _i m_{2,i} \beta _i$ where $m_{2,i}\leq m_{1,i} \leq m_{0,i}$ for all $i$ and such that of all possible choices for $\tau _2$, the coefficient $m_{2,2}$ is minimal. (Notice of course that $m_{2,1}=m_{1,1}$ for any choice of $\tau _2$ by our choice of $\tau _1$.)

Continue in this manner until obtaining a root $\tau _n  \in \Phi^+ \cap  (\Phi^w)^- $ such that $\tau_n = \sum _i m_{n,i} \beta _i$ where $m_{n,i} \leq m_{n-1,i} \leq \cdots \leq m_{0,i} \leq 0$ for all $i$ and such that if there is a root $\sum _i m_{i} \beta _i \in \Phi^+ \cap  (\Phi^w)^-$ with $m_i \leq m_{n,i}$ for all $i$, then $\sum _i m_{i} \beta _i =\tau _n$.

We rename $\tau_n$ as $\tau _{J,I,w}$. If $\alpha \in \Phi^+ \cap  (\Phi^w)^-$, then $\alpha +\tau _{J,I,w} \notin  \Phi^+ \cap  (\Phi^w)^-$ by the previous paragraph. But $ \Phi^+ \cap  (\Phi^w)^-$ is closed under addition, so it must be that $\alpha +\tau _{J,I,w} \notin  \Phi$. This is part $(i)$ of the lemma.

Part $(ii)$ follows from $\tau _{J,I,w} \in  \Phi^+$.

For part $(iii)$, notice that $\mathfrak{C}_{^w{\bf P}_J} \nsubseteq \mathcal{H}_j$ implies that  $\alpha _j \notin [J^w]$. Therefore $\tau _{J,I,w} \notin [J^w]$, and in particular, $m_{n,k} < 0$ for some $k >|J|$.

Let $a \in \,^wA_J^+(t)$ where $t>1$. Then $|\beta _i(a)|=1$ if $i \leq |J|$ and $|\beta _i(a)|\geq t$ for all $i >|J|$. Therefore $|\beta_i (a)|^{m_{n,i}} \leq 1$ for all $i$ (since $m_{n,i}\leq 0$) , and  
\begin{align*} |\tau _{J,I,w} (a)| &= \prod_{v \in S}|\tau _{J,I,w} (a_v)|_v \\
&= \prod_{v \in S}\Big|\sum_im_{n,i}\beta_i (a_v)\Big|_v \\
&= \prod_{v \in S}\prod_i(|\beta_i (a_v)|_v)^{m_{n,i}} \\
&= \prod_i|\beta_i (a)|^{m_{n,i}} \\
&\leq|\beta_k (a)|^{m_{n,k}} \\
&\leq t^{m_{n,k}} 
\end{align*}

\end{proof}

\subsection{Proximity to integer points} Our proof will rely on identifying certain points in $G$ that are close to points in ${\bf G}(\mathcal{O}_S)$ (Lemma~\ref{l:close} below), identifying certain points in $G$ that are far from points in ${\bf G}(\mathcal{O}_S)$  (Lemma~\ref{l:far} below), and then contrasting these two identifications.

\begin{lemma}\label{l:21}
Suppose $X \se {\bf G}(K)$ is a finite set. Then there is a bounded set $B \se G$ such that $X {\bf G}(\mathcal{O}_S) \se {\bf G}(\mathcal{O}_S)B$.

\end{lemma}

\begin{proof}

For $x \in X$, we let $\Gamma _x =  {\bf G}(\mathcal{O}_S) \cap  x{\bf G}(\mathcal{O}_S) x^{-1}$. Since $x \in {\bf G}(K)$, there is a finite set $\{y_1,y_2,\ldots ,y_k\} $ of right coset representatives for $\Gamma _x$ in $ x{\bf G}(\mathcal{O}_S) x^{-1}$.

Thus, $x {\bf G}(\mathcal{O}_S) =x{\bf G}(\mathcal{O}_S) x^{-1}x=\cup_i \Gamma_x y_i x \se {\bf G}(\mathcal{O}_S) \{y_1,y_2,\ldots ,y_k\}x $. 

\end{proof}

As a consequence of the previous lemma we have

\begin{lemma}\label{l:close}
There is some $C>0$ such that any point in $$ {\bf G}(\mathcal{O}_S)F^{-1}{\bf G}(\mathcal{O}_S)F{\bf G}(\mathcal{O}_S)UMW^{-1}$$ is
within distance $C$ of a point in ${\bf G}(\mathcal{O}_S)$.
\end{lemma}

\begin{proof}
Since $\bf U$ is unipotent and $\bf M$ is $K$-anisotropic, it follows that there is some compact set $B \se UM$ such that $({\bf{UM}})(\mathcal{O}_S)B=UM$. Thus, $ {\bf G}(\mathcal{O}_S)F^{-1}{\bf G}(\mathcal{O}_S)F{\bf G}(\mathcal{O}_S)UMW^{-1}$ is contained in \newline ${\bf G}(\mathcal{O}_S) F^{-1}{\bf G}(\mathcal{O}_S)F{\bf{G}}(\mathcal{O}_S)BW^{-1}$ and the lemma follows from the previous lemma.

\end{proof}

For $\tau \in \Phi $ let ${\bf A}_{\tau}$ be the kernel of $\tau  $ in $\bf A$. Let $$A(\tau, t)=\{\, a\in A \mid |\tau (a)| \geq t  \,\}$$ and fix $a_\tau \in A$ such that $|\tau (a_\tau)|>1$.

\begin{lemma}\label{l:half} There is some $C>0$ such that for any $ k_0 \in \mathbb{N}$, there is some $t_0>1$ such that  the Hausdorff distance between $\cup_{k \geq k_0}{\bf A}(\mathcal{O}_S){ A}_{\tau} (a_\tau) ^k$  and $A(\tau,t_0)$ is at most $C$.
\end{lemma}

\begin{proof}
Since $\mathcal{O}_S \se K_w$ is bounded if $w \notin S$, we have that
${\bf A}(\mathcal{O}_S) \leq {\bf A}(K_w)$ is bounded. Hence the image of ${\bf A}(\mathcal{O}_S)$ under the map $g \mapsto |\tau(g)|_w$ is bounded and therefore is trivial. 

For any $x \in K$, the product over all valuations $v$ of $K$ of $|x|_v$ equals $1$, so it follows that $|\tau (a)|=1$ for any 
$a \in {\bf A}(\mathcal{O}_S)$.

Notice also that $|\tau (a)|=1$ for any $a\in A_\tau$, so ${\bf A}(\mathcal{O}_S){ A}_{\tau} \se A(\tau,1)\cap A(-\tau,1)$ and
the lemma will follow for $t_0=|\tau ((a_\tau)^{k_0})|$ if we establish that ${\bf A}(\mathcal{O}_S){ A}_{\tau} $  is a finite Hausdorff distance from $A(\tau,1)\cap A(-\tau,1)$. This essentially follows from a dimension count. 

The group $ A_\tau$ is quasi-isometric to Euclidean space of dimension $|S| ({\text{rank}}_K({\bf A})-1)$. Dirichlet's units theorem gives us that the dimension of ${\bf A}(\mathcal{O}_S)$ equals $(|S|-1)  {\text{rank}}_K({\bf A})$ and that the dimension of ${\bf A}_\tau(\mathcal{O}_{S})$ equals $(|S|-1)  ({\text{rank}}_K({\bf A})-1) $.

Since ${\bf A}(\mathcal{O}_S) \cap { A}_{\tau}={\bf A}_\tau(\mathcal{O}_{S})$, it follows that the dimension of  ${\bf A}(\mathcal{O}_S){ A}_{\tau} $ equals $$  ( |S|-1)  {\text{rank}}_K({\bf A}) +|S| ({\text{rank}}_K({\bf A})-1) - (|S|-1) ({\text{rank}}_K({\bf A})-1)$$
The above number is $|S|  {\text{rank}}_K({\bf A})  -1$, which is the dimension of $A(\tau,1)\cap A(-\tau,1)$. Therefore, ${\bf A}(\mathcal{O}_S){ A}_{\tau} $  is a finite Hausdorff distance from $A(\tau,1)\cap A(-\tau,1)$ which proves the lemma.
\end{proof}

We will use the previous lemma to establish the following 

\begin{lemma}\label{l:far}
Suppose $I,J \subsetneq \Delta$, that $w \in W$, and $w \notin {\bf P}_J(K)$. Then for any $C>0$, there exists some $t>1$ such that any $g \in U_{w,-}MA(\tau _{J,I,w},t)$ is distance at least $C$ from  ${\bf G}(\mathcal{O}_S)$.
\end{lemma}

\begin{proof}

Choose $\gamma \in {\bf U}_{(\tau _{J,I,w})}(\mathcal{O}_S)$ with $\gamma \neq 1$.

For $k\in \mathbb{N}$, let $$O_k=\{\,u \in { U}_{(\tau _{J,I,w})} \mid d(u,1)\leq \frac 1k\,\}$$ and let $$F_k=\{\,g\in {G} \mid g^{-1} \gamma g \in O_k \,\}$$ so that $F_{k+1} \se F_k$.

For $k$ sufficiently large, $O_k \cap  {\bf G}(\mathcal{O}_S)=1$, which implies $F_k \cap  {\bf G}(\mathcal{O}_S)=\emptyset$, and in fact $\lim_{k \to \infty}d(F_k \,,\,  {\bf G}(\mathcal{O}_S))= \infty$. 
Let $m$ be such that  the distance between $F_{m}$ and ${\bf G}(\mathcal{O}_S)$ is sufficiently large.
 
 Note that $\lim _k (a_{\tau _{J,I,w}})^{-k}\gamma (a_{\tau _{J,I,w}})^{k}=1$, so there is some $k_0 \in \mathbb{N}$ such that $(a_{\tau _{J,I,w}})^k \in F_m$ if $k \geq k _0$.
 
Let $\alpha \in \Phi ^+ \cap (\Phi ^w)^-$. By $(i)$ of Lemmas~\ref{l:root1} and~\ref{l:root2}, commutators of elements in $U_{({\alpha})}$ with elements in $U_{({\tau _{J,I,w}})}$ are contained in $U_{({\alpha + \tau _{J,I,w}})}=1$. That is, the group $U_{w,-}$ commutes with $U_{({\tau _{J,I,w}})}$,  and in particular, with $\gamma$. 

 Notice that ${ A}_{\tau _{J,I,w}}$ also commutes with $\gamma \in U_{(\tau _{J,I,w})}$. Therefore, if $g \in U_{w,-}{ A}_{\tau _{J,I,w}}$ and $k \geq k_0$, then $$(a_{\tau _{J,I,w}})^{-k}g^{-1}\gamma g(a_{\tau _{J,I,w}})^{k}=(a_{\tau _{J,I,w}})^{-k} \gamma (a_{\tau _{J,I,w}})^{k}\in O_m$$ so $g (a_{\tau _{J,I,w}})^k \in F_m$. 
 
The distance between  $ {\bf G}(\mathcal{O}_S)$ and $\lambda F_n$ for any $\lambda \in {\bf G}(\mathcal{O}_S)$ equals the distance between $F_n$ and $\lambda ^{-1} {\bf G}(\mathcal{O}_S) = {\bf G}(\mathcal{O}_S) $. Therefore, the union over ${k\geq k_0}$ of the sets $$U_{w,-} {\bf M}(\mathcal{O}_S){\bf A}(\mathcal{O}_S){ A}_{{\tau _{J,I,w}}} (a_{\tau _{J,I,w}})^k= {\bf M}(\mathcal{O}_S){\bf A}(\mathcal{O}_S)U_{w,-}{ A}_{{\tau _{J,I,w}}} (a_{\tau _{J,I,w}})^k$$ is a sufficiently large distance from  ${\bf G}(\mathcal{O}_S)$. 

Using Lemma~\ref{l:half}, we have for some $t>1$ that there is a sufficiently large distance between ${\bf G}(\mathcal{O}_S)$ and $U_{w,-} {\bf M}(\mathcal{O}_S)A(\tau _{J,I,w},t) $, and thus that there is a sufficiently large distance between ${\bf G}(\mathcal{O}_S)$ and $U_{w,-} {\bf M}(\mathcal{O}_S)A(\tau _{J,I,w},t) B$ where $B$ is a given compact set. Precisely, since $\bf M$ is $K$-anisotropic, we choose $B \se M$ to be a compact fundamental domain for ${\bf M}(\mathcal{O}_S)$. Our lemma follows since elements of $B$ commute with those in $A$. 
 
\end{proof}

Notice that in the above proof, the properties of $\tau_{J,I,w}$ are used to find an integral unipotent element ($\gamma$) that commutes with the unipotent group $U_{w,-}$. Thus, if $U_{w,-}$ were replaced with the trivial group in the above lemma, we would be free to apply the resulting statement to any root $\tau \in \Phi$. That is, the proof of the preceding lemma simplifies to prove the following

\begin{lemma}\label{l:far2}
Suppose $\tau \in \Phi $. Then for any $C>0$, there exists some $t>1$ such that any $g \in A(\tau ,t)$ is distance at least $C$ from  ${\bf G}(\mathcal{O}_S)$.
\end{lemma}

As an immediate consequence of Lemma~\ref{l:far} we have the following

\begin{lemma}\label{l:halfup}
Suppose $I,J \subsetneq \Delta$, that $w \in W$, and $w \notin {\bf P}_J(K)$.
 For any 
 bounded set $B \se G$, there is some $s>1$ such that the sets $UMA(\tau_{J,I,w},1) B$ and $^wUMA(-\tau_{J,I,w},s) B$ are disjoint.
\end{lemma}

\begin{proof}
Recall that $MA$ normalizes $U$ and $^w U$, elements of $A$ commute with elements in $M$, and that the inverse of an element in $A(-\tau_{J,I,w},s)$ is contained in $A(\tau_{J,I,w},s)$.

Thus, we can multiply given elements from each of the sets in question on the left by inverses of elements in $^wUMA(-\tau_{J,I,w},s)$, and on the right by inverses of elements in $B$, to see that the lemma follows from showing that the sets $ ^wUUMA(\tau_{J,I,w},s) $ and $BB^{-1}$ are disjoint for some $s>1$.

Recall that $U_{w,+} \,,\, U_{w,-} \leq U$, that $U_{w,+}U_{w,-}=U$, and that $U_{w,+} \leq \,^wU$. Thus, $ ^wUU=\,  ^wUU_{w,+}U_{w,-}= \,^wUU_{w,-}$, and after multiplying on the left by the inverses of elements in $^wU$, we are left to prove that the sets $ U_{w,-}MA(\tau_{J,I,w},s) $ and $^wU BB^{-1}$ are disjoint for some $s>1$. 

But $^w \bf U$ is a unipotent $K$-group, so there is some compact set $B_w \se \,^wU$ such that $^wU=(\,^w {\bf U}) (\mathcal{O}_S)B_w$. Thus, we need to show that  the sets $ U_{w,-}MA(\tau_{J,I,w},s) $ and $(\,^w {\bf U}) (\mathcal{O}_S) (B_w BB^{-1})$ are disjoint for some $s>1$. This follows from Lemma~\ref{l:far}.

\end{proof}

\subsection{Disjointness of distinct parabolic regions}
The goal of this subsection is to prove Lemma~\ref{l:31}, which will quickly imply that distinct parabolic regions are --- after removing a neighborhood of their boundaries --- disjoint.

Given $\gamma_1,\gamma_2,\gamma_3 \in {\bf G}(\mathcal{O}_S)$ and $f_1,f_2 \in F$, let $ p_2wp_1 \in {\bf P}(K) W  {\bf P}(K)$ be such that $\gamma_3^{-1}f_2^{-1}\gamma_2 f_1 \gamma_1 = p_2wp_1$. Let $p_i=u_im_ia_i$ for $u_i \in {\bf U}(K)$, $m_i \in {\bf M}(K)$, and $a_i \in {\bf A}(K)$.

Given $\gamma_3^{-1}f_2^{-1}\gamma_2 f_1 \gamma_1$, our choice of group elements $p_2$, $w$, $p_1$, $u_2$, $m_2$, $a_2$, $u_1$, $m_1$, and $a_1$  will be fixed for Lemmas~\ref{l:chamber}
 and~\ref{l:noname}.

\begin{lemma}\label{l:chamber} 
Suppose $I,J \subsetneq \Delta$, that $w \in W$, and $w \notin {\bf P}_J(K)$.
Suppose $s>0$ is given. Then there is some $t>1$ (independent of $\gamma_3^{-1}f_2^{-1}\gamma_2 f_1 \gamma_1$) such that 
$$a_2 \,  w a_1  A _J^+ (t) w^{-1} \se A(-\tau_{J,I,w},s)$$
\end{lemma}

\begin{proof}
Let $u_2 =u_2 ^- (u_2^+)^{-1}$ where $u_2^- \in U_{w,-}$ and  $u_2^+ \in U_{w,+}$. Since $MA$ normalizes $U_{w,+}$ we have $$^{a_2^{-1}m_2^{-1}}u_2^+ \in U_{w,+} \leq U_{(\Phi^w)^+}$$ 
Notice that if $\alpha \in \Phi$ and $v \in  {\bf U}_{(\alpha)} $, then $^{w^{-1}}v \in {\bf U}_{(\alpha ^{w^{-1}})}$. Therefore $$^{(m_2a_2w)^{-1}}u_2^+ = \,^{w^{-1}}(^{a_2^{-1}m_2^{-1}}u_2^+) \in U_{\Phi ^+}=U$$
 It follows that $${(u_1m_1)^{-1}} [^{(m_2a_2w)^{-1}}u_2^+] \in UM$$ and thus
$$^{a_1^{-1}}[{(u_1m_1)^{-1}} [^{(m_2a_2w)^{-1}}u_2^+]] \in UM$$
since $A$ normalizes $UM$.

By Lemma~\ref{l:close}, the following point is a bounded distance from ${\bf G}(\mathcal{O}_S)$:
\begin{align*}\gamma_3^{-1}f_2^{-1}\gamma_2 & f_1 \gamma_1\, ^{a_1^{-1}} [{(u_1 m_1)^{-1}} ^{(m_2a_2w)^{-1}}u_2^+] w^{-1}\\ 
& =  p_2wp_1\, ^{a_1^{-1}}[{(u_1m_1)^{-1}} [^{(m_2a_2w)^{-1}}u_2^+]]w^{-1} \\
& =u_2m_2a_2wu_1m_1a_1 \,^{a_1^{-1}}[{(u_1m_1)^{-1}} [^{(m_2a_2w)^{-1}}u_2^+]] w^{-1}\\
& =u_2m_2a_2w(u_1m_1)(u_1m_1)^{-1}  [^{(m_2a_2w)^{-1}}u_2^+]a_1w^{-1} \\
& =u_2m_2a_2w \, ^{(m_2a_2w)^{-1}}u_2^+ a_1w^{-1} \\
& = u_2 u_2^+ m_2a_2w a_1w^{-1} \\
& =(u_2^-m_2)a_2  w a_1 w^{-1}
\end{align*}

By Lemma~\ref{l:far}, there is some $r>1$ that is independent of $a_2  (w a_1w^{-1})\in A $ and such that  $|\tau _{J,I,w}(a_2 w a_1w^{-1} )| < r$.

By $(iii)$ of Lemma~\ref{l:root1} there is some $t>1$ such that $|\tau _{J,I,w} (a)| < 1/sr$ for any $a \in \,^wA_J^+(t)$. 

Therefore, $$ |\tau _{J,I,w} ( a_2 w a_1 w^{-1} a)|= |\tau _{J,I,w} ( a_2 w a_1 w^{-1})| |\tau _{J,I,w} ( a)| < 1/s$$
and thus $|-\tau _{J,I,w}(a_2 w a_1 w^{-1} a )| > s$.

\end{proof}

\begin{lemma}\label{l:noname} 
Suppose $I,J \subsetneq \Delta$. For any bounded set $B \se G$ there is some $t>1$ (independent of $\gamma_3^{-1}f_2^{-1}\gamma_2 f_1 \gamma_1$) such that if  $\gamma_3^{-1}f_2^{-1}\gamma_2 f_1 \gamma_1 \in {\bf G}(\mathcal{O}_S) F^{-1} {\bf G}(\mathcal{O}_S) F {\bf G}(\mathcal{O}_S) $ is not contained in ${\bf P}_J(K)$, then the sets $\gamma_3^{-1}f_2^{-1}\gamma_2 f_1 \gamma_1 UMA^+_J(t)B$  and $UMA^+_I(t)B$ are disjoint.

 \end{lemma}

\begin{proof}
If $\gamma_3^{-1}f_2^{-1}\gamma_2 f_1 \gamma_1$ is not contained in ${\bf P}_J(K)$ then $w \notin {\bf P}_J(K)$.

The group $A$ commutes with $M$ and it normalizes $U$ and $^wU$. The group $M$ normalizes $U$ and $^wM=M$. Thus,
\begin{align*}
\gamma_3^{-1}f_2^{-1}\gamma_2 f_1 \gamma_1 UMA^+_J(t)B & = p_2wp_1  UMA^+_J(t)B \\
&= p_2w  UMa_1A^+_J(t)B \\
&= p_2\,^wUMwa_1A^+_J(t)B \\
&= u_2\,^wUMa_2wa_1A^+_J(t)B 
\end{align*}

By Lemma~\ref{l:halfup}, there is some $s>1$  such that $UMA(\tau_{J,I,w},1)[B \cup wB]$  is disjoint from $^wUMA(-\tau_{J,I,w},s)[B\cup wB]$.

By Lemma~\ref{l:root1}, $ A^+_I(1) \se  A(\tau_{J,I,w},1)$. By  Lemma~\ref{l:chamber}, there is some $t>1$ such that   $a_2wa_1A^+_J(t)w^{-1} \se A(-\tau_{J,I,w},s)$. Therefore,  $u_2^{-1}UMA^+_I(1)B$ is disjoint from $^wUMa_2wa_1A^\uparrow_J(t)B$ which proves the lemma.

\end{proof}

\begin{lemma}\label{l:noname2} 
Suppose $I,J \subsetneq \Delta$ with $|I|=|J|$ and $I \neq J$. For any bounded set $B \se G$ there is some $t>1$ (independent of $\gamma_3^{-1}f_2^{-1}\gamma_2 f_1 \gamma_1$) such that if  $\gamma_3^{-1}f_2^{-1}\gamma_2 f_1 \gamma_1 \in {\bf G}(\mathcal{O}_S) F^{-1} {\bf G}(\mathcal{O}_S) F {\bf G}(\mathcal{O}_S) $ is contained in ${\bf P}_J(K)$ and ${\bf P}_I(K)$, then the sets $\gamma_3^{-1}f_2^{-1}\gamma_2 f_1 \gamma_1 A^+_J(t)B$  and $A^+_I(t)B$ are disjoint.

 \end{lemma}

  \begin{proof}

We let  $p=\gamma_3^{-1}f_2^{-1}\gamma_2 f_1 \gamma_1$ Notice that $p \in {\bf P}_{I \cap J}(K)$ so that $p=uma$ where $u \in U_{\Phi(I\cap J)^+}$, $m \in M_{I\cap J}$, and $a \in A_{I \cap J}^+$.

Elements of $A^+_I \leq A^+_{I\cap J} $ commute with $ma$, and they normalize   $U_{\Phi(I\cap J)^+}$. Therefore, $A^+_I(t)^{-1}p \se U_{\Phi(I\cap J)^+} p A^+_I(t)^{-1} $. Hence, 
$$A^+_I(t)^{-1}p A^+_J(t) B \se U_{\Phi(I\cap J)^+} p A^+_I(t)^{-1} A^+_J(t) B$$ and the lemma will follow if we show that $p ^{-1} U_{\Phi(I\cap J)^+} B B^{-1}$ is disjoint from $A^+_I(t)^{-1} A^+_J(t) $.

Since ${\bf U}_{\Phi(I\cap J)^+}$ is a unipotent $K$-group, there is some compact set $B_{I\cap J} \se U_{\Phi(I\cap J)^+} $ such that ${\bf U}_{\Phi(I\cap J)^+}(\mathcal{O}_S)B_{I\cap J}= U_{\Phi(I\cap J)^+}$. Therefore,
$p ^{-1} U_{\Phi(I\cap J)^+} B B^{-1} $ is contained in $${\bf G}(\mathcal{O}_S) F {\bf G}(\mathcal{O}_S) F^{-1} {\bf G}(\mathcal{O}_S) {\bf U}_{\Phi(I\cap J)^+}(\mathcal{O}_S)B_{I\cap J}B B^{-1}$$
and thus is contained in a metric neighborhood of ${\bf G}(\mathcal{O}_S) $ by Lemma~\ref{l:21}.

The set $[I] \cap \Phi(J)^+$ is nonempty since $J $ does not contain $I$, and we choose $\tau \in [I] \cap \Phi(J)^+ $. Thus $A^+_I(t)^{-1} A^+_J(t)  \se A(\tau, t)$ and the lemma follows from Lemma~\ref{l:far2}.

 \end{proof}

\begin{lemma}\label{l:28}  Suppose $I,J \subsetneq \Delta$ with $|I|=|J|$.
Let $B \se G$ be a bounded set. There is a $t>1$ such that if $f_2^{-1}\gamma _2 f_1 \in F^{-1} {\bf G}(\mathcal{O}_S) F$ and either $f_2^{-1}\gamma _2 f_1 \notin {\bf P}_J(K)$, $f_2^{-1}\gamma _2 f_1 \notin {\bf P}_I(K)$,  or $J \neq I$, then the sets $f_2^{-1}\gamma_2 f_1 {\bf P}_J(\mathcal{O}_S)A_J^+(t)B$  and $ {\bf P}_I(\mathcal{O}_S)A_I^+(t)B$ are disjoint. \end{lemma}

\begin{proof}
By the previous two lemmas there is some $t>1$ such that for any  $\gamma _3 \in  {\bf P}_I(\mathcal{O}_S)$ and  $\gamma _1 \in  {\bf P}_J(\mathcal{O}_S)$ we have that $\gamma _3 ^{-1}f_2^{-1}\gamma _2 f_1 \gamma _1 A^+ _J(t)B$ is disjoint from  $A^+ _I(1)B$ as long as either $\gamma _3 ^{-1}f_2^{-1}\gamma _2 f_1 \gamma _1 \notin {\bf P}_J(K)$, $\gamma _1 ^{-1}f_1^{-1}\gamma _2^{-1} f_2 \gamma _3 \notin {\bf P}_I(K)$, or $I \neq J$.

If $I=J$ and $\gamma _3 ^{-1}f_2^{-1}\gamma _2 f_1 \gamma _1 \in {\bf P}_J(K)$, then 
$f_2^{-1}\gamma _2 f_1 \in \gamma _3 {\bf P}_J(K) \gamma _1^{-1}= {\bf P}_J(K)$.

If $I=J$ and $\gamma _1 ^{-1}f_1^{-1}\gamma _2^{-1} f_2 \gamma _3 \in {\bf P}_I(K)$, then 
$f_1^{-1}\gamma _2^{-1} f_2 \in \gamma _1 {\bf P}_I(K) \gamma _3^{-1}= {\bf P}_I(K)$, and hence 
$f_2^{-1}\gamma _2 f_1 \in {\bf P}_I(K)$.

\end{proof}

At this point, we have done most of the work that was required in this subsection. The next 3 lemmas provide some cosmetic reformulation of what we have done.

We let $A^\uparrow _I(t)=\{\, a \in A_I \mid |\alpha (a)|_v\geq t {\text{ if }} \alpha \in \Delta -I\ {\text{ and }} v \in S\}$.

\begin{lemma}\label{l:300}
For $I\subsetneq \Delta$ there is some bounded set $B_{A,I} \se A_I$ containing $1$ such that for any $t>1$, $A_I^+(t)\se {\bf A}_I(\mathcal{O}_S)A_I^\uparrow(\sqrt[|S|]{t})B_{A,I}$.
\end{lemma}

\begin{proof}
If $|S|=1$, then $A_I^\uparrow(t)=A_I^+(t)$ and the lemma follows.

If $|S|>1$, then let $a \in A_I^+(t)$ and choose $w \in S$ such that $|\alpha (a)|_w\geq|\alpha (a)|_v $ for $v \in S-\{w\}$.
 
 By Lemma~\ref{l:tproj}, there is some $a _0 \in {\bf A}_I(\mathcal{O}_S)$ such that $|\alpha (a_0 a )|_v \geq \sqrt[|S|]{t} $ for all  $v \in S-\{w\}$ and such that the distance between $|\alpha (a_0 a )|_w$ and $ \sqrt[|S|]{t} $ is uniformly bounded. Thus, there is some bounded $a_b \in A_I$ such that $a_0aa_b \in A_I^\uparrow(\sqrt[|S|]{t})$. Hence, $a \in a_0^{-1}A_I^\uparrow(\sqrt[|S|]{t})a_b^{-1}$.
 
\end{proof}

For $c>0$ we let $B_I(c)=\{\, u \in U_{\Phi(I)^+} \mid ||u||\leq c \,\}$. Note that $B_I(c)$ is compact, and since ${\bf U}_{\Phi(I)^+}$ is unipotent, there is some $c_0$  such that ${\bf U}_{\Phi(I)^+}(\mathcal{O}_S) B_I(c_0) = U_{\Phi(I)^+}$. We let $B_I=B(c_0)$. Notice that if $a \in A_I^\uparrow(1)$ and $b \in B_I$ then $a^{-1}ba \in B_I$ so that $B_IA_I^\uparrow(t) \se A_I^\uparrow(t)B_I$ when $t>1$.

\begin{lemma}\label{l:ray}
Given $I\subsetneq \Delta$ There is some bounded set $B \se G$ such that if $t>1$  then
$$U_{\Phi(I)^+}{\bf M}_I(\mathcal{O}_S)A_I^+(t)\,\,\se \,\, {\bf P}_{I}(\mathcal{O}_S)  A_J^\uparrow(\sqrt[|S|]{t})B$$
\end{lemma}

\begin{proof}
Because $U_{\Phi(I)^+}$ is normalized by  ${\bf M}_I(\mathcal{O}_S) \leq M_I$ and $ {\bf A}_I(\mathcal{O}_S)\leq A_I$, Lemma~\ref{l:300} yields the following inclusions of sets
\begin{align*}
U_{\Phi(I)^+}{\bf M}_I(\mathcal{O}_S) A_I^+(t) 
& \,\se\, U_{\Phi(I)^+} {\bf M}_I(\mathcal{O}_S) {\bf A}_I(\mathcal{O}_S)A_I^\uparrow(\sqrt[|S|]{t})B_{A,I}\\
& = {\bf M}_I(\mathcal{O}_S) {\bf A}_I(\mathcal{O}_S)U_{\Phi(I)^+} A_I^\uparrow(\sqrt[|S|]{t})B_{A,I}\\
& = {\bf M}_I(\mathcal{O}_S) {\bf A}_I(\mathcal{O}_S){\bf U}_{\Phi(I)^+}(\mathcal{O}_S)B_I A_I^\uparrow(\sqrt[|S|]{t})B_{A,I}\\
& \,\se\, {\bf P}_{I}(\mathcal{O}_S)A_I^\uparrow(\sqrt[|S|]{t})B_I B_{A,I}\\ 
\end{align*}

\end{proof}

And now we have the lemma that this subsection was devoted to in

\begin{lemma}\label{l:31}
 Let $I ,J \subsetneq \Delta $ with $|I|=|J|$. Let $B \se G$ be a bounded set. There is a $t>1$ such that if $f_2^{-1}\gamma _2 f_1 \in F^{-1} {\bf G}(\mathcal{O}_S) F$ and either $f_2^{-1}\gamma _2 f_1 \notin {\bf P}_J(K)$,  $f_2^{-1}\gamma _2 f_1 \notin {\bf P}_I(K)$,  or $J \neq I$, then  the sets $f_2^{-1}\gamma_2 f_1 U_{\Phi(J)^+}{\bf M}_J(\mathcal{O}_S)A_J^+(t)B$  and $ U_{\Phi(I)^+}{\bf M}_I(\mathcal{O}_S)A_I^+(t)B$ are disjoint. \end{lemma}

\begin{proof}
As $A_J^\uparrow(\sqrt[|S|]{t})$ is contained in $A_J^+(t)$, the proof is a straightforward combination of Lemmas~\ref{l:28} and~\ref{l:ray}.
\end{proof}

\subsection{Coarse stabilization of parabolic regions under parabolic translations}
In the next lemma we will prove that translating the parabolic region associated to ${\bf P}_J$ by elements of ${\bf P}_J(K) \, \cap \, F^{-1}{\bf G}(\mathcal{O}_S)F$ stabilize the parabolic region up to a bounded Hausdorff distance.

\begin{lemma}\label{l:32} Let $J \subsetneq \Delta$. There is a bounded set $B \se G$ such that if $t>1$ and $f_2^{-1}\gamma  f_1 \in F^{-1} {\bf G}(\mathcal{O}_S) F$ with $f_2^{-1}\gamma  f_1 \in {\bf P}_J(K)$, then the set $f_2^{-1}\gamma f_1 U_{\Phi(J)^+}{\bf M}_J(\mathcal{O}_S)A_J^+(t)$  
is contained in $U_{\Phi(J)^+}{\bf M}_J(\mathcal{O}_S)A_J^+(t)B$.
\end{lemma}

\begin{proof}
Notice that $^{f_2^{-1}\gamma  f_1}{\bf P}_J={\bf P}_J$ and hence $^{\gamma  f_1}{\bf P}_J=\,^{f_2}{\bf P}_J$. 

Let $\gamma _0 $ be a fixed element of $ {\bf{G}}(\mathcal{O}_S)$ with $f_2^{-1}\gamma_0  f_1 \in {\bf P}_J(K)$. Then $^{\gamma _0  f_1}{\bf P}_J=\,^{f_2}{\bf P}_J$, and by letting  $ \lambda =\gamma \gamma_0^{-1}$ we have that  $$^{\lambda f_2}{\bf P}_J=\,^{\lambda \gamma _0  f_1}{\bf P}_J=\,^{\gamma f_1}{\bf P}_J=\,^{ f_2}{\bf P}_J$$ which implies that $\lambda \in (\,^{f_2}{\bf P}_J)(O_S)$.

Let $\Lambda _1 \se {\bf {P}}_J(\mathcal{O}_S)$ be a finite index subgroup such that $(\gamma _0 f_1)\Lambda _1 (\gamma _0 f_1)^{-1}$ is contained in $ (\, ^{\gamma _0 f_1} {\bf {P}}_J)(\mathcal{O}_S)$ and let $g_1,\ldots , g_m$ be a set of right coset representatives for $\Lambda _1$ in  ${\bf {P}}_J(\mathcal{O}_S)$.

Let $\Lambda _2 \se (\,^{f_2}{\bf {P}}_J)(\mathcal{O}_S)$ be a finite index subgroup such that $ f_2^{-1}\Lambda _2  f_2$ is contained in $ {\bf {P}}_J(\mathcal{O}_S)$ and let $h_1,\ldots , h_\ell$ be a set of right coset representatives for $\Lambda _2$ in  $ (\,^{f_2}{\bf {P}}_J)(\mathcal{O}_S)$.

Each $g_i \in  {\bf P}_J(K)$,  
$^{ \gamma _0  f_1}{\bf P}_J=\,^{f_2}{\bf P}_J$, and $h_j \in (\,^{f_2}{\bf P}_J)(K)$ for all $j$. Therefore, $f_2^{-1}h_j(\gamma _0 f_1)g_i $ normalizes, and hence is contained in, $ {\bf P}_J(K)$. We choose a bounded set $B' \se P_J$ such that $$\bigcup _{i,j} f_2^{-1}h_j(\gamma _0 f_1)g_i  \, \se \, B'$$ As in the comments preceding Lemma~\ref{l:ray}, we may assume that $$B'A_J^\uparrow(\sqrt[|S|]{t}) \se A_J^\uparrow(\sqrt[|S|]{t}) B'$$

We have the following inclusion of sets:
\begin{align*}
f_2^{-1}\gamma f_1 {\bf P}_J&(\mathcal{O}_S) A_J^\uparrow(\sqrt[|S|]{t}) \\
& = f_2^{-1}\lambda (\gamma _0 f_1) {\bf P}_J(\mathcal{O}_S) A_J^\uparrow(\sqrt[|S|]{t}) \\
& = \bigcup _i f_2^{-1}\lambda (\gamma _0 f_1) \Lambda _1g_i A_J^\uparrow(\sqrt[|S|]{t}) \\
& = \bigcup _i f_2^{-1}\lambda (\gamma _0 f_1) \Lambda _1 (\gamma _0 f_1)^{-1} (\gamma _0 f_1)g_i A_J^\uparrow(\sqrt[|S|]{t}) \\
&  \,\se\,  \bigcup _i f_2^{-1}\lambda (\, ^{\gamma _0 f_1} {\bf {P}}_J)(\mathcal{O}_S) (\gamma _0 f_1) g_i A_J^\uparrow(\sqrt[|S|]{t}) \\
& = \bigcup _i f_2^{-1}\lambda (\, ^{f_2} {\bf {P}}_J)(\mathcal{O}_S) (\gamma _0 f_1) g_i A_J^\uparrow(\sqrt[|S|]{t}) \\
& = \bigcup _i f_2^{-1} (\, ^{f_2} {\bf {P}}_J)(\mathcal{O}_S) (\gamma _0 f_1) g_i A_J^\uparrow(\sqrt[|S|]{t}) \\
& = \bigcup _{i,j} f_2^{-1} \Lambda _2 h_j (\gamma _0 f_1) g_i A_J^\uparrow(\sqrt[|S|]{t}) \\
& = \bigcup _{i,j} f_2^{-1} \Lambda _2 f_2 f_2^{-1}h_j (\gamma _0 f_1) g_i A_J^\uparrow(\sqrt[|S|]{t}) \\
& = \bigcup _{i,j} {\bf {P}}_J(\mathcal{O}_S) f_2^{-1}h_j (\gamma _0 f_1) g_i  A_J^\uparrow(\sqrt[|S|]{t}) \\
& \, \se \, {\bf {P}}_J(\mathcal{O}_S) B' A_J^\uparrow(\sqrt[|S|]{t}) \\
& \,\se\, {\bf {P}}_J(\mathcal{O}_S)  A_J^\uparrow(\sqrt[|S|]{t}) B '
\end{align*}

Let $B$ be as in Lemma~\ref{l:ray}. Then \begin{align*} f_2^{-1}\gamma f_1 U_{\Phi(J)^+}{\bf M}_J(\mathcal{O}_S)A_J^+(t) & \,\se \, f_2^{-1}\gamma f_1 {\bf P}_{I}(\mathcal{O}_S)  A_I^\uparrow(\sqrt[|S|]{t})B \\
& \, \se \, {\bf {P}}_J(\mathcal{O}_S)  A_J^\uparrow(\sqrt[|S|]{t}) B 'B\\
& \, \se \, U_{\Phi(J)^+}{\bf M}_J(\mathcal{O}_S)A_J^+(t) B 'B
\end{align*}
\end{proof}

\subsection{Proof of Proposition~\ref{p:prune}}

\begin{proof}
For part $(i)$, notice that $\bf M$ is $K$-anisotropic. Therefore there is a bounded set $B_M \se M$ such that $M={\bf M}(\mathcal{O}_S)B_M$. If $B$ is as in  Theorem~\ref{t:red}, then we let $B_0=B_MB$. Thus, \begin{align*} {\bf G}(\mathcal{O}_S)FU{\bf M}(\mathcal{O}_S)A_\emptyset^+B_0 
&={\bf G}(\mathcal{O}_S)FU{\bf M}(\mathcal{O}_S)A_\emptyset^+B_MB \\
&= {\bf G}(\mathcal{O}_S)FU{\bf M}(\mathcal{O}_S)B_MA_\emptyset^+B \\
&= {\bf G}(\mathcal{O}_S)FU MA_\emptyset^+B \\
& =G 
\end{align*}

For $(ii)$, suppose $B_n$ and $N_n$ are given. Let ${\bf Q}=\,^{\gamma _1 f_1}{\bf P}_J$ and ${\bf Q'}=\,^{\gamma _2 f_2}{\bf P}_I$. Assume that ${\bf Q}\neq {\bf Q'}$ and that $|I|=|J|=n$.

Let $\gamma f \in \Lambda _{\bf Q}$. If $t>1$, then Lemma~\ref{l:32} shows that there is some bounded set $B \se G$ such that $\gamma  f U_{\Phi(J)^+}{\bf M}_J(\mathcal{O}_S)A_J^+(t)B_n$ is contained in $\gamma _1 f_1 U_{\Phi(J)^+}{\bf M}_J(\mathcal{O}_S)A_J^+(t)BB_n$. Therefore,
 $R _{\bf Q} (t)B_n$ is contained in $\gamma _1 f_1 U_{\Phi(J)^+}{\bf M}_J(\mathcal{O}_S)A_J^+(t)BB_n$.
 
Similarly, 
$R _{\bf Q'} (t)B_n$ is contained in  $\gamma _2 f_2 U_{\Phi(I)^+}{\bf M}_I(\mathcal{O}_S)A_I^+(t)BB_n$.

If $ f_2^{-1}\gamma _2^{-1} \gamma _1 f_1\in {\bf P}_J (K)$ and $I=J$ then $ ^{f_2^{-1}\gamma _2^{-1} \gamma _1 f_1}{\bf P}_J = {\bf P}_I$ which contradicts that  ${\bf Q}\neq {\bf Q'}$. Thus, we can apply
Lemma~\ref{l:31}  to find some $t_n>1$ such that $f_2^{-1}\gamma _2^{-1} R_{\bf Q}(t_n)B_nB_N$ is 
disjoint from $f_2^{-1}\gamma _2^{-1} R_{\bf Q '}(t_n)B_nB_N$ where $B_N$ is a neighborhood of $1\in G$ of radius $N_n$.

For part $(iii)$, notice in the above that we could choose $t_n$ to be arbitrarily large. 

Let $w \in W$ represent the longest element of the Weyl group  so that $\tau_{J,I,w}$ is the highest root with respect to $\Phi^+$ and $U_{w,-}=U$. Then by Lemma~\ref{l:far}, there is some $t_n>1$ such that
$ U_{\Phi(I)^+}A(\tau _{J,I,w},t_n)$ is arbitrarily far from  ${\bf G}(\mathcal{O}_S)$. Hence, $ U_{\Phi(I)^+}A^+_I(t_n)$, and thus $ {\bf M}_I(\mathcal{O}_S)U_{\Phi(I)^+}A^+_I(t_n)$, is arbitrarily far from  ${\bf G}(\mathcal{O}_S)$. Then by Lemma~\ref{l:21}, $ U_{\Phi(I)^+}{\bf M}_I(\mathcal{O}_S)A^+_I(t_n)$ is arbitrarily far from  $F^{-1}{\bf G}(\mathcal{O}_S){\bf G}(\mathcal{O}_S)$ which proves this part of the proposition.

For parts $(iv)$ and $(v)$, let $ I \subsetneq \Delta$. Let $\gamma \in {\bf G}(\mathcal{O}_S)$, $f \in F$, $u \in U_{\Phi(I)^+}$, $m \in {\bf M}_I(\mathcal{O}_S)$, $a \in A_I^+$, and $b \in B_n$. Furthermore, assume that $|\alpha (a)|<2t_n$ for some $\alpha \in \Delta -I$.  Let $J = I \cup \alpha$.

There is a bounded neighborhood of the identity $B_{I, \alpha} \se A$ depending on $2t_n$ such that $a \in A^+_{J}B_{I, \alpha}$. Thus, $$\gamma f u m a b \in \gamma f U_{\Phi(I)^+}{\bf M}_I(\mathcal{O}_S)A^+_{J}B_{I, \alpha}B_n$$

The $K$-group ${\bf U} _{\Phi(I)^+\cap[J]^+} $ is unipotent, so there is some bounded set $B_J \se U _{\Phi(I)^+\cap[J]^+} \se M_J$ containing the identity such that ${\bf U} _{\Phi(I)^+\cap[J]^+}  (\mathcal{O}_S) B_J =U _{\Phi(I)^+\cap[J]^+} $. Recall that  ${\bf M}_I(\mathcal{O}_S)$ normalizes $U_{\Phi(I)^+} $. Therefore, 
\begin{align*} U_{\Phi(I)^+}{\bf M}_I(\mathcal{O}_S) & ={\bf M}_I(\mathcal{O}_S)U_{\Phi(I)^+} \\
&=  {\bf M}_I(\mathcal{O}_S)U_{\Phi(J)^+} U_{\Phi(I)^+\cap[J]^+} \\
& \,\se \, {\bf M}_J(\mathcal{O}_S)U_{\Phi(J)^+} U_{\Phi(I)^+\cap[J]^+} \\
& = U_{\Phi(J)^+} {\bf M}_J(\mathcal{O}_S) U_{\Phi(I)^+\cap[J]^+} \\
& = U_{\Phi(J)^+} {\bf M}_J(\mathcal{O}_S) {\bf U}_{\Phi(I)^+\cap[J]^+} (\mathcal{O}_S) B_J \\
& = U_{\Phi(J)^+} {\bf M}_J(\mathcal{O}_S)  B_J
\end{align*}

Since $B_J \se M_J$ commutes with $A^+_J$, we have  $$\gamma f u m a b \in \gamma f U_{\Phi(J)^+} {\bf M}_J(\mathcal{O}_S)  A^+_{J}B_JB_{I, \alpha}B_n$$

If we let $B'$ be the product over $I$ and $\alpha$ of the sets $B_JB_{I, \alpha}$, then we can let $B_{n+1}=B'B_n$.

In the case when $|I|=|\Delta| -1$, notice that for $|J|=|\Delta|$ the groups $U_{\Phi(J)^+}$ and  $A^+_J$ are trivial and ${\bf M}_J={\bf G}$, so $$\gamma f u m a b \in \gamma f {\bf G}(\mathcal{O}_S)  B_{|\Delta|}$$ The lemma follows after enlarging $ B_{|\Delta|}$ in view of Lemma~\ref{l:21}.

For $(vi)$,  suppose ${\bf Q}= \, ^{\gamma f}{\bf P}_I$ for some $\gamma  \in {\bf G}(\mathcal{O}_S)$ and $f \in F$. As in the proof of $(ii)$,  Lemma~\ref{l:32} implies that there is a bounded set $B \se G$ such that
$R _{\bf Q} (t_n)B_n$ is contained in $\gamma  f U_{\Phi(I)^+}{\bf M}_I(\mathcal{O}_S)A_I^+(t_n)BB_n$, and there is an obvious quasi-isometry from  $\gamma  f U_{\Phi(I)^+}{\bf M}_I(\mathcal{O}_S)A_I^+(t_n)BB_n$ to the space $ U_{\Phi(I)^+}{\bf M}_I(\mathcal{O}_S)A_I^+$ that satisfies the proposition.

\end{proof}

\end{document}